\documentclass[11pt,a4paper,leqno,noamsfonts]{amsart}
\usepackage[british]{babel}
\usepackage[dvipsnames]{xcolor}
\usepackage{graphicx,pifont,comment}
\usepackage[utopia]{mathdesign}
\usepackage[a4paper,top=3.2cm,bottom=3.2cm,left=3.5cm,right=3.5cm,marginparwidth=60pt]{geometry}
\usepackage[utf8]{inputenc}
\usepackage{braket,caption,comment,mathtools,stmaryrd,enumitem}
\usepackage[usestackEOL]{stackengine}
\usepackage{multirow,booktabs,microtype,relsize,soul}
\usepackage[foot]{amsaddr}
\usepackage[colorlinks,bookmarks]{hyperref} %
      \hypersetup{colorlinks,%
            citecolor=britishracinggreen,%
            filecolor=black,%
            linkcolor=cobalt,%
            urlcolor=RawSienna}
      \numberwithin{equation}{section}
\usepackage[capitalise]{cleveref}

        \usepackage{euflag}

\allowdisplaybreaks

\newcommand{\SE}{{\mathscr E}}
\newcommand{\SF}{{\mathscr F}}
\newcommand{\SG}{{\mathscr G}}

\newcommand{\ST}{{\mathscr T}}
\newcommand{\SU}{{\mathscr U}}

\newcommand{\BA}{{\mathbb{A}}}

\newcommand{\BE}{{\mathbb{E}}}

\newcommand{\BL}{{\mathbb{L}}}

\newcommand{\BP}{{\mathbb{P}}}
\newcommand{\BQ}{{\mathbb{Q}}}

\newcommand{\BT}{{\mathbb{T}}}

\newcommand{\BZ}{{\mathbb{Z}}}

\newcommand{\CC}{{\mathcal C}}

\newcommand{\CE}{{\mathcal E}}
\newcommand{\CF}{{\mathcal F}}
\newcommand{\CG}{{\mathcal G}}
\newcommand{\CH}{{\mathcal H}}

\newcommand{\CK}{{\mathcal K}}
\newcommand{\CL}{{\mathcal L}}
\newcommand{\CM}{{\mathcal M}}

\newcommand{\CT}{{\mathcal T}}

\DeclareSymbolFont{cmarrows}{OMS}{cmsy}{m}{n}
\SetSymbolFont{cmarrows}{bold}{OMS}{cmsy}{b}{n}
\DeclareMathSymbol{\cmminus}{\mathbin}{cmarrows}{"00}

\DeclareMathSymbol{\leftrightarrow}{\mathrel}{cmarrows}{"24}
\DeclareMathSymbol{\leftarrow}{\mathrel}{cmarrows}{"20}
   
\DeclareMathSymbol{\rightarrow}{\mathrel}{cmarrows}{"21}
   \let\to=\rightarrow
\DeclareMathSymbol{\mapstochar}{\mathrel}{cmarrows}{"37}
   \def\mapsto{\mapstochar\rightarrow}

\DeclareSymbolFont{usualmathcal}{OMS}{cmsy}{m}{n}
\DeclareSymbolFontAlphabet{\mathcal}{usualmathcal}
\DeclareMathAlphabet\BCal{OMS}{cmsy}{b}{n} 
\newcommand*{\myInfty}{{\scriptstyle\infty}} 

\makeatletter
\newcommand{\mylabel}[2]{#2\def\@currentlabel{#2}\label{#1}}
\makeatother

\newcommand*{\defeq}{\mathrel{\vcenter{\baselineskip0.5ex \lineskiplimit0pt
                     \hbox{\scriptsize.}\hbox{\scriptsize.}}}=}
                     
\definecolor{cornellred}{rgb}{0.7, 0.11, 0.11}
\definecolor{britishracinggreen}{rgb}{0.0, 0.26, 0.15}
\definecolor{cobalt}{rgb}{0.0, 0.28, 0.67}

\usepackage[colorinlistoftodos]{todonotes} 

\newcommand{\into}{\hookrightarrow}
\newcommand{\onto}{\twoheadrightarrow}
\newcommand{\HH}{\mathrm{H}}

\newcommand{\OO}{\mathscr O}

\newcommand{\boldit}[1]{\boldsymbol{#1}}

\newcommand{\crit}{\operatorname{crit}}
\newcommand{\Rcrit}{\mathbf{R}\mathrm{crit}}

\DeclareMathOperator{\vir}{\mathrm{vir}}
\DeclareMathOperator{\coker}{coker}

\DeclareMathOperator{\Perf}{Perf}
\DeclareMathOperator{\Cone}{\mathrm{Cone}}

\DeclareMathOperator{\RR}{\mathbf{R}}

\DeclareMathOperator{\Spec}{Spec}
\DeclareMathOperator{\Hilb}{Hilb}
\DeclareMathOperator{\Quot}{Quot}

\DeclareMathOperator{\Coh}{Coh}
\DeclareMathOperator{\coh}{\mathsf{coh}}
\DeclareMathOperator{\qcoh}{\mathsf{qcoh}}
\DeclareMathOperator{\QCoh}{QCoh}

\DeclareMathOperator{\id}{id}
\DeclareMathOperator{\pr}{pr}

\DeclareMathOperator{\Supp}{Supp}
\DeclareMathOperator{\derived}{\mathbf{D}}
\DeclareMathOperator{\bounded}{\mathrm{b}}
\DeclareMathOperator{\lissetale}{\mathrm{lis-\acute{e}t}}
\DeclareMathOperator{\etale}{\mathrm{\acute{e}t}}
\DeclareMathOperator{\Boxplus}{\mathlarger{\mathlarger{\mathlarger{\boxplus}}}}

\DeclareMathOperator{\Sets}{\mathsf{Sets}}
\DeclareMathOperator{\Mod}{Mod}
\DeclareMathOperator{\cdga}{\mathsf{cdga}_{\Bbbk}^{\leqslant0}}
\DeclareMathOperator{\Cat}{\mathsf{Cat}}
\DeclareMathOperator{\Fun}{\mathsf{Fun}}
\DeclareMathOperator{\Spaces}{\mathsf{Spaces}}
\DeclareMathOperator{\dSt}{\mathsf{dSt}_{\Bbbk}}
\DeclareMathOperator{\dSch}{\mathsf{dSch}_{\Bbbk}}
\DeclareMathOperator{\dArt}{\mathsf{dArt}_{\Bbbk}}
\DeclareMathOperator{\dAff}{\mathsf{dAff}_{\Bbbk}}

\DeclareMathOperator{\St}{\mathsf{St}_{\Bbbk}}

\DeclareMathOperator{\Hom}{Hom}

\DeclareMathOperator{\Ext}{Ext}
\DeclareMathOperator{\lHom}{\mathscr{H}\kern-0.3em\mathit{om}}
\DeclareMathOperator{\RRlHom}{\mathbf{R}\kern-0.025em\mathscr{H}\kern-0.3em\mathit{om}}
\DeclareMathOperator{\RHom}{{\mathbf{R}\mathrm{Hom}}}
\DeclareMathOperator{\lExt}{{\mathscr{E}\kern-0.2em\mathit{xt}}}

\usepackage[all]{xy}
\usepackage{tikz}
\usepackage{tikz-cd}
\usepackage{adjustbox}
\usepackage{rotating}

\usetikzlibrary{matrix,shapes,intersections,arrows,decorations.pathmorphing,decorations.markings}
\tikzset{commutative diagrams/.cd,
mysymbol/.style={start anchor=center,end anchor=center,draw=none}}
\newcommand\MySymb[2][\square]{%
  \arrow[mysymbol]{#2}[description]{#1}}
\tikzset{
shift up/.style={
to path={([yshift=#1]\tikztostart.east) -- ([yshift=#1]\tikztotarget.west) \tikztonodes}
}
}



\theoremstyle{definition}

\newtheorem*{lemma*}{Lemma}
\newtheorem*{theorem*}{Theorem}
\newtheorem*{example*}{Example}
\newtheorem*{fact*}{Fact}
\newtheorem*{notation*}{Notation}
\newtheorem*{definition*}{Definition}
\newtheorem*{prop*}{Proposition}
\newtheorem*{remark*}{Remark}
\newtheorem*{corollary*}{Corollary}

\newtheorem*{conventions*}{Conventions}

\newtheorem{definition}{Definition}[section]

\newtheorem{remark}[definition]{Remark}
\newtheorem{warning}[definition]{Warning}

\newtheoremstyle{thm} 
        {3mm}
        {3mm}
        {\slshape}
        {0mm}
        {\bfseries}
        {.}
        {1mm}
        {}
\theoremstyle{thm}
\newtheorem{theorem}[definition]{Theorem}
\newtheorem{corollary}[definition]{Corollary}
\newtheorem{lemma}[definition]{Lemma}
\newtheorem{prop}[definition]{Proposition}

\newtheorem{thm}{Theorem}

\newtheorem*{Acknowledgments*}{Acknowledgments}

\newcommand{\fl}{\mathsf{flat}}
\newcommand{\uni}{\mathrm{uni}}

\DeclareMathOperator{\RSpec}{\mathbf{Spec}}
\DeclareMathOperator{\RQuot}{\mathbf{R}Quot}

\DeclareMathOperator{\DPerf}{\mathbf{Perf}}

\newcommand{\RQuotlX}[1]{\RQuot^{[#1]}_{X}(\CE)}
\newcommand{\DPerflX}[1]{\mathbf{Perf}^{[#1]}(X)}

\newcommand{\XX}{\boldit{X}}
\newcommand{\YY}{\boldit{Y}}
\newcommand{\ZZ}{\boldit{Z}}
\newcommand{\UU}{\boldit{U}}
\newcommand{\QQ}{\boldit{Q}}
\newcommand{\PP}{\boldit{P}}
\newcommand{\hh}{\boldit{h}}

\title[Derived hyperquot schemes]{Derived hyperquot schemes}
\author{Sergej Monavari, Emanuele Pavia, Andrea T. Ricolfi}

\begin{document}
\begin{abstract}
We define a derived enhancement of the \emph{hyperquot scheme} (also known as \emph{nested Quot scheme}), which classically parametrises flags of quotients of a perfect coherent sheaf on a projective scheme. We prove it is representable by a derived scheme, and we compute its global tangent complex. As an application, we provide a natural obstruction theory on the classical hyperquot scheme. The latter recovers the virtual fundamental class recently constructed by the first and third author in the context of the  enumerative geometry of hyperquot schemes on smooth projective curves.
\end{abstract}

\maketitle
{\hypersetup{linkcolor=black}\tableofcontents}

\section{Introduction}

\subsection{Overview}
Let $X$ be a projective scheme over an algebraically closed field $\Bbbk$ of characteristic 0. Fix a coherent sheaf $\CE$ on $X$ and an integer $l \geqslant 1$. The \emph{hyperquot scheme}\footnote{The hyperquot scheme is also known in the literature as the \emph{nested Quot scheme}, see \cite{monavari2024hyperquot} and the references therein.} 
\[
\Quot_X^{[l]}(\CE)
\]
is the fine moduli space
whose $B$-points, for $B$ an arbitrary $\Bbbk$-scheme, 
classify isomorphism classes of $B$-flat families of quotients
\[
\begin{tikzcd}
\CE_B \arrow[two heads]{r} & 
\CT_1 \arrow[two heads]{r} & 
\CT_2 \arrow[two heads]{r} &
\cdots \arrow[two heads]{r} & 
\CT_l
\end{tikzcd}
\]
in $\Coh(X\times B)$. When $\CE=\OO_X$, this space coincides with the \emph{nested Hilbert scheme}, a widely studied and yet still pretty mysterious moduli space. If $l=1$ (i.e.~when the space is the usual Quot scheme \cite{Grothendieck_Quot,Nit}), it is classically known (see e.g.~\cite[Thm.~6.4.9]{fga_explained}) that the deformation and (higher) obstruction spaces at a point $p = [\CE \onto \CT]\in \Quot_X^{[1]}(\CE)$ are given by the cohomology groups
\begin{align}\label{eqn: obstr}
\Hom_{X}(\CK, \CT),\,\,\,\, \Ext^{\geqslant 1}_X(\CK, \CT),
\end{align}
where $\CK=\ker(\CE \onto \CT)$. In particular, even when the underlying scheme $X$ is smooth,  the presence of nontrivial  obstruction spaces reflects the singularity of $ \Quot_X^{[1]}(\CE)$.

Derived algebraic geometry is the `correct' modern framework to formulate and suitably deal with moduli problems \emph{and} their deformation theory at once. In fact, according to Kontsevich's \emph{hidden smoothness principle} \cite{kontsevich_94}, under reasonable assumptions, moduli spaces which are singular as schemes should be enhanced to smooth \emph{derived} moduli spaces. A key example of this philosophy is the \emph{derived Quot scheme} 
\[
\RQuot_X^{[1]}(\CE),
\]
two versions\footnote{See \cite{BSY_derived1,borisov2023shiftedsymplecticstructuresderived} for a general machinery clarifying the relation among the two constructions.} of which have already been constructed: the first by Ciocan-Fontanine--Kapranov in the framework of dg-schemes \cite{DerivedQuot}, the second by Adhikari in the language of commutative differential graded algebras \cite{DerivedQuotNew}, both under the smoothness assumption on $X$. For both enhancements, the \emph{derived} tangent complex at a point $[\CE \onto \CT]\in \RQuot_X^{[1]}(\CE)$ is represented by $\RHom_X(\CK, \CT)$, a perfect complex of vector spaces whose cohomologies recover the usual deformation-obstruction spaces \eqref{eqn: obstr}.

We point out that enhancements of \emph{smooth} schemes are also of great interest: we refer the reader to Jiang's works \cite{derived-grassmannians-schur} and \cite{derived-projectivizations} for the theory of derived Grassmannians and derived projectivisations of bundles.

\smallbreak
In this paper we extend the results of \cite{DerivedQuot,DerivedQuotNew} by constructing  a derived enhancement $ \RQuot_X^{[l]}(\CE)$ of the classical hyperquot scheme $ \Quot_X^{[l]}(\CE)$ in the general case of flags of quotients of arbitrary length $l\geqslant 1$, and in doing so we relax the assumptions on $X$, only assuming it to be projective (not necessarily smooth), and under the mild requirement that $\CE$ be \emph{perfect} as a coherent sheaf (cf.~\Cref{sec:Amplifications} for the definition of perfectness and for a strategy to relax the projectivity assumption on $X$). As an application of our novel construction, we  explicitly compute the  \emph{global} derived tangent complex of $\RQuot_X^{[l]}(\CE) $ in terms of (universal) homological data attached to $(X,\CE,l)$. Furthermore, by functoriality of tangent complexes, we exhibit a (global) obstruction theory on $\Quot_X^{[l]}(\CE)$, which was lacking in the literature.\footnote{To the best of our knowledge, even a `punctual' version of such obstruction theory, namely a \emph{tangent-obstruction theory} in the sense of \cite[Ch.~6]{fga_explained}, is new in this level of generality.}

\smallbreak
In the next section, we state our results in greater detail.

\subsection{Main results}
We work in the language of commutative differential graded algebras, which provides a suitable framework for derived algebraic geometry in characteristic 0, following \cite{hag1}. A \emph{derived moduli problem} is a functor
\begin{equation*}
\begin{tikzcd}
\cdga \arrow{r}{\boldit{F}} & \Spaces
\end{tikzcd}
\end{equation*}
from the $\myInfty$-category of connective commutative differential graded algebras to the $\myInfty$-category of spaces. We say that $\boldit{F}$ is a \emph{derived enhancement} of a classical  moduli problem 
\begin{equation*}
\begin{tikzcd}
\mathsf{Alg}_{\Bbbk}\arrow{r}{F} & \Sets
\end{tikzcd}
\end{equation*}
if, for $A\in \mathsf{Alg}_{\Bbbk} \subsetneq \cdga$ a classical commutative $\Bbbk$-algebra, we have that $\boldit{F}(A)$ is homotopy equivalent to $F(A)$.

\subsubsection{Derived enhancement}
Our first main result is the construction of a derived enhancement of $\Quot_X^{[l]}(\CE)$.

\begin{thm}[{\Cref{thm:truncation}}]
\label{MAIN-THM-DERIVED-ENHANCEMENT}
Let $X$ be a projective scheme, $\CE \in \Coh(X)$ a perfect coherent sheaf, and $l$ a positive integer. Then the hyperquot scheme $\Quot_X^{[l]}(\CE)$ admits a natural derived enhancement $\RQuot_X^{[l]}(\CE)$.
\end{thm}

The idea behind the proof of \Cref{MAIN-THM-DERIVED-ENHANCEMENT} is the following. First, we define the  derived stack $\DPerf^{[l]}(X)_{\CE/}$ of length $l+1$ flags of perfect complexes on $X$ with fixed source $\CE$ (this uses perfectness of $\CE$). Then we define the \emph{derived hyperquot functor}
\[
\begin{tikzcd}
\RQuot_X^{[l]}(\CE) \colon \cdga \arrow{r} & \Spaces
\end{tikzcd}
\]
as the (open) subfunctor of $\DPerf^{[l]}(X)_{\CE/}$ parametrising flat families of surjective morphisms (in the derived sense, cf.~\Cref{sec:flatness-surj}). We prove in \Cref{sec:geometricity-results} a series of results on the \emph{local geometricity} of several auxiliary derived stacks, including $\DPerf^{[l]}(X)_{\CE/}$, leading to the local geometricity of the derived stack $\RQuot_X^{[l]}(\CE)$.
Then, we show that $\RQuot_X^{[l]}(\CE)$ has `no higher homotopy' (cf.~{\Cref{thm:truncation}}), and that in particular it is a derived enhacement of the classical hyperquot functor $\mathsf{Quot}_X^{[l]}(\CE)$. From this, we deduce in \Cref{RQuot-is-derived-scheme} that $\RQuot_X^{[l]}(\CE)$ is representable by a \emph{derived scheme}.

\subsubsection{Tangent complex}
Our second main result is an explicit (in terms of universal data) computation of the global tangent complex of the derived enhancement of \Cref{MAIN-THM-DERIVED-ENHANCEMENT}. A point-wise version of this result, for the case $l=1$, was proved by Adhikari \cite{DerivedQuotNew} (see also \cite{DerivedQuot}).

\begin{thm}[\Cref{prop:tang-cplx-Q}]\label{MAIN-THEOREM-TG-CPLX}
Let $X$ be a projective scheme, $\CE \in \Coh(X)$ a perfect coherent sheaf, and $l$ a positive integer. The tangent complex of $\QQ = \RQuot_X^{[l]}(\CE)$ is
\[
\BT_{\QQ} \simeq \Cone\left(\bigoplus_{1\leqslant i \leqslant l} \RRlHom_{\boldit{\pi}} (\SE_i^{\uni},\SE_i^{\uni})\xrightarrow{~~~\Upsilon~~~}  \displaystyle\bigoplus_{1\leqslant i \leqslant l} \RRlHom_{\boldit{\pi}}(\SE_{i-1}^{\uni},\SE_{i}^{\uni})\right),
\]
where 
\begin{equation*}
\begin{tikzcd}
\SE_0^{\uni} = \CE_{\QQ}\arrow[two heads]{r} & \SE_1^{\uni} \arrow[two heads]{r} & \SE_2^{\uni} \arrow[two heads]{r} & \cdots \arrow[two heads]{r} & \SE_l^{\uni}    
\end{tikzcd} 
\end{equation*}
is the universal flag over $\QQ \times X$, $\boldit{\pi}\colon \QQ \times X \to \QQ$ is the projection and $\Upsilon$ is defined in \eqref{Upsilon-map}.
\end{thm}

The main ingredients in the proof of \Cref{MAIN-THEOREM-TG-CPLX} are the computation of the relative tangent complex of the `source-and-target map'
\[
\begin{tikzcd}
\DPerf^{[1]}(X) \arrow{rr}{\langle s,t\rangle} && \DPerf(X)^{\times 2},   
\end{tikzcd}
\]
cf.~\Cref{prop:TANGENT-source-and-target}, and the functoriality of tangent complexes in derived geometry, applied to a suitable map, namely the map $\hh \colon \QQ \to \DPerf(X)^{\times l}$ sending a flag to underlying the $l$-tuple of quasicoherent sheaves.

\subsubsection{Obstruction theory}
Our third main result describes a natural obstruction theory, in the sense of Behrend--Fantechi \cite{BFinc}, on the classical hyperquot scheme.

\begin{thm}[\Cref{main-theorem-obs-th}]\label{MAIN-THEOREM-OBS-TH}
Let $X$ be a projective scheme, $\CE \in \Coh(X)$ a perfect coherent sheaf, and $l$ a positive integer. Set $Q=\Quot_X^{[l]}(\CE)$. Let $\CE_{Q} \onto \ST_\bullet^{\uni}$ be the universal flag, and let $\pi\colon Q \times X \to Q$ be the projection. Then $Q$ admits an obstruction theory
\[
\BE_Q = \Cone \left(
\displaystyle\bigoplus_{1\leqslant i \leqslant l} \RRlHom_{\pi}(\ST_{i-1}^{\uni},\ST_{i}^{\uni})^\vee \xrightarrow{~~~~~} \displaystyle\bigoplus_{1\leqslant i \leqslant l} \RRlHom_{\pi}(\ST_{i}^{\uni},\ST_{i}^{\uni})^\vee
\right)[-1]
\xrightarrow{~~~~~} \BL_Q.
\]
\end{thm}

The strategy for the proof of \Cref{MAIN-THEOREM-OBS-TH} is the following. By \Cref{MAIN-THEOREM-TG-CPLX}, we know the cotangent complex $\BL_{\QQ} = \BT_{\QQ}^\vee$ explicitly, and we confirm that $\BE_Q\simeq \iota^\ast \BL_{\QQ}$, where $\iota \colon Q \into \QQ$ is the canonical inclusion. On the other hand, the canonical morphism $\iota^\ast \BL_{\QQ} \to \BL_Q$ is naturally an obstruction theory by \Cref{prop:OT-truncated}.

The obstruction theory of \Cref{MAIN-THEOREM-OBS-TH} generalises the classical obstruction theory of the ordinary Quot scheme $\Quot_X^{[1]}(\CE)$ (see e.g.~\cite[Thm.~4.6]{Gillam}), and appears to be new in the literature. When $X$ is a smooth projective curve, $\BE_Q$ is perfect in $[-1,0]$ and therefore, by the work of Behrend--Fantechi  \cite{BFinc}, it induces a \emph{virtual fundamental class} $[Q]^{\vir} \in A_\ast (Q)$, which we checked to coincide with the virtual fundamental class constructed by the first and third author in \cite{monavari2024hyperquot}, cf.~\Cref{sec:curve-case}.

\subsection{Amplifications}
\label{sec:Amplifications}
We now briefly comment on our working assumptions. Projectivity of $X$ is used in a twofold way: first of all, to ensure the representability of the (classical) hyperquot functor, and secondly, properness is needed to ensure that the derived stack $\DPerf(X)$ of perfect complexes on $X$ is locally geometric (cf.~\Cref{thm:perf_locally_geometric} and \Cref{rmk:perfect-stack}\ref{perf-1}). Perfectness of $\CE$ means that $\CE$, when viewed as an object of $\derived_{\coh}^{\bounded}(X)$, lies in the image of the canonical functor $\Perf(X) \to \derived_{\coh}^{\bounded}(X)$, which  ensures that $\CE$ defines a $\Bbbk$-point of $\DPerf(X)$. Note that, if $X$ is smooth, any $\CE \in \Coh(X)$ is automatically perfect, as the functor $\Perf(X) \to \derived_{\coh}^{\bounded}(X)$ is, in this case, an exact equivalence of triangulated categories \cite[\href{https://stacks.math.columbia.edu/tag/0FDC}{Tag 0FDC}]{stacks-project}. 

In the remainder of this section we discuss what adjustments are needed in order to extend our results to the case where $X$ is only quasiprojective or over a general base not necessarily of characteristic 0.

\subsubsection{Quasiprojective case}
\label{subsec:quasi-proj}
Let $X$ be a  quasiprojective $\Bbbk$-scheme, $\CE \in \Coh(X)$ a coherent sheaf, and $l$ a positive integer. The classical hyperquot functor attached to $(X,\CE,l)$, recalled in \Cref{sec:nested-quot-classical}, is still representable if one further assumes that
\[
\begin{tikzcd}
    \Supp(\CT_i) \arrow[hook]{r} & X\times B \arrow{r} & B
\end{tikzcd}
\]
is proper for all $i=1,\ldots,l$. The construction of a derived enhancement $\QQ$ of the hyperquot scheme $Q=\Quot_X^{[l]}(\CE)$ carries over to the quasiprojective case, just as in Adhikari's paper \cite{DerivedQuotNew}, if one takes into account the following technical modifications:
\begin{itemize}
    \item [\mylabel{quasi-1}{(1)}] One should replace all the stacks of perfect complexes and flags thereof (cf.~\Cref{subsec:derived-flags}) by their pseudo-perfect version \cite[Def.~2.2]{DerivedQuotNew}, and confirm their local geometricity. 
    \item [\mylabel{quasi-2}{(2)}] In the definition of the derived hyperquot functor \eqref{derived-quot-functor}, one should require each $\CE_i$ to have proper support over $A$, which means that $\Supp(\CE_i)$, defined to be the union of the classical supports $\Supp(\CH^j(\CE_i))$, is proper over $\Spec \HH^0(A)$ for every $i=1,\ldots,l$.
\end{itemize}

As soon as the above changes are made, the construction of the derived enhancement of $\Quot_X^{[l]}(\CE)$ carries over verbatim, and no further technical modification is needed for the material presented in \Cref{sec:proof-of-main-thm} regarding the cotangent complex.

\subsubsection{Arbitrary base}
One can naturally generalise our constructions over an arbitrary base, possibly of positive characteristic. In this setting,  commutative differential graded algebras \emph{do not} capture in general the correct homotopy theory of commutative rings. However, one can still define a theory of derived algebraic geometry over a base ring $A$ (possibly of positive characteristic) by replacing the $\myInfty$-category of commutative differential graded $A$-algebras with the $\myInfty$-category $\mathsf{Alg}_{\BE_{\myInfty}}(\Mod_{HA})$ of \textit{$\BE_{\myInfty}$-$HA$-ring spectra} (see e.g.~\cite[Ch.~7]{ha}). This allows one to define a derived moduli functor $\RQuotlX{l}$ for any scheme $X$ defined over an arbitrary $\BE_{\myInfty}$-ring  and for any  perfect sheaf $\CE$ over $X$.

Since the proof of \cref{MAIN-THM-DERIVED-ENHANCEMENT} relies on the representability of the classical hyperquot scheme, which holds under very mild assumptions of the base ring \cite{Grothendieck_Quot}, and on the local geometricity of the derived stack $\DPerf(X)$, which for $X$ a classical proper scheme holds over any derived G-ring (\cite[Notation 1.5.9 and Thm.~4.4.1]{Pandit2011ModuliPI}), our result holds over any derived G-ring as well. This includes every classical G-ring -- in particular any field (of arbitrary characteristic), any Dedekind domain of characteristic $0$, and any localisation or finitely generated algebra obtained from G-rings.

\subsection{Relation to existing literature}
Hyperquot schemes and their obstruction theories are ubiquitous in enumerative geometry and theoretical physics. For $X$ a smooth projective curve, the hyperquot scheme $\Quot_X^{[l]}(\OO_X^{\oplus r})$ is a compactification of the space of maps from $X$ to the partial flag variety, and  was effectively exploited to compute the quantum cohomology ring of partial flag varieties, see e.g.~\cite{CF_quantum_duke, CF_quantum_cohom_IMRN, CF_quantum_Trans, Chen_quantum}. The motives of the  hyperquot schemes  $\Quot_X^{[l]}(\CE)$ were computed, in the smooth unobstructed case, by the first and third author in \cite{MR_nested_Quot, monavari2024hyperquot}, see also Chen for $X=\BP^1$ and $\CE=\OO_X^{\oplus r}$ \cite{Chen-hyperquot},  Cheah's work on nested Hilbert schemes \cite{MR2716513,MR1616606,MR1612677} and \cite{Mon_double_nested,double-nested-1} on double nested Hilbert schemes. 

Quot schemes of smooth surfaces and their virtual invariants have also been studied extensively in recent years, see for instance \cite{MR4220750,zbMATH07344916,Lim:2020aa,stark2024quot} and the references therein. The case $\CE=\OO_X$, corresponding to nested Hilbert schemes on surfaces, has also drawn great attention in enumerative geometry \cite{zbMATH07184851,Sheshmani-Yau,zbMATH07159378,zbMATH07270076,BFT_fisica,BFT_math}. In these works, the authors consider a perfect obstruction theory which differs from the one of \Cref{MAIN-THEOREM-OBS-TH}. Nevertheless, we believe it should be recovered from a  nested version of the \emph{derived Hilbert scheme} of \cite{DerivedHilbert}.

Quot schemes of points on smooth 3-folds are the key characters in higher rank Donaldson--Thomas theory. Here a subtlety arises: the `local model' for these spaces, namely $\Quot_{\BA^3}(\OO^{\oplus r},\mathsf{points})$, is a global critical locus $\crit(f)$ \cite{BR18}, and therefore carries the \emph{critical obstruction theory} $\iota^\ast \BL_{\Rcrit(f)} \to \BL_{\crit(f)}$, where $\iota \colon \crit(f) \into \Rcrit(f)$ is the inclusion. It is proved in \cite[Thm.~C]{d-critical_quot} that, when $r=1$, this agrees with the obstruction theory (coming from derived geometry in a suitable sense) of $\Hilb(\BA^3,\mathsf{points})$ viewed as a moduli space of ideal sheaves; a related conjecture can be found in \cite[Conj.~2.12]{cazzaniga2020framed}, where the Quot scheme of points on $X=\BA^m$ for $m\geqslant 3$ was proved to be equivalent to the moduli space of \emph{framed sheaves} on $\BP^m$. See also \cite{BR18,Virtual_Quot,FMR_K-DT,cazzaniga2020higher,ricolfi2019motive,fasola2023tetrahedron} for a virtual (and refined) enumerative study of these Quot schemes, and the series of works \cite{LocalDT,Ricolfi2018,DavisonR} exploring, again via Quot schemes, the `fixed-curve contributions' to (rank 1) Donaldson--Thomas invariants.

\subsection{Organisation of contents}
The paper is organised as follows. \Cref{sec:background} is devoted to background in derived algebraic geometry: derived stacks and their truncation (Sections \ref{sec:derivedstacks} and \ref{sec:truncation}), cotangent complexes and obstruction theories (Sections \ref{sec:cotg-complex} and \ref{sec:ob-theories}), quasicoherent sheaves on derived stacks (\Cref{sec:qcoh-derivestacks}). In \Cref{sec:derived-nested-quot}, after reviewing the definition of the classical hyperquot functor (\Cref{sec:nested-quot-classical}), we introduce the derived stack $\DPerf^{[l]}(\XX)$ parametrising flags of perfect complexes of length $l+1$ on a derived stack $\XX$, along with other `auxiliary' stacks (\Cref{subsec:derived-flags}). We prove crucial geometricity results about them in \Cref{sec:geometricity-results}, and we use these results in \Cref{subsec:derived-quot-functor} to confirm that the derived hyperquot functor, introduced in \eqref{derived-quot-functor}, is a derived scheme. In \Cref{sec:source-target-tangent} we identify the relative tangent complex to the `source and target map' $\DPerf^{[1]}(\XX) \to \DPerf(\XX)^{\times 2}$, and we use this computation crucially in \Cref{sec:cotangent-complex-quot} to prove \Cref{MAIN-THEOREM-TG-CPLX}. \Cref{sec:proof-of-OT} is devoted to the proof of \Cref{MAIN-THEOREM-OBS-TH}, and finally in \Cref{sec:curve-case} we discuss the special case of curves, comparing our results to other results in the literature.

\begin{conventions*}
We work over a fixed algebraically closed field $\Bbbk$ of characteristic 0.  We use cohomological conventions throughout (for graded algebras, complexes of sheaves, etc.). For the derived algebraic geometry concepts we will need in this paper (recalled in \Cref{sec:background}), we follow mainly To\"{e}n--Vezzosi \cite{hag1,hag2}, To\"{e}n \cite{Toen-higher-and-derived}, Gaitsgory--Rozenblyum \cite{studyindag1,studyindag2} and Lurie \cite{dagviii}. We often omit deriving functors, e.g.~we write $f^\ast$, resp.~$f_\ast$, instead of $\mathbf{L}f^\ast$, resp.~$\RR f_\ast$, for $f\colon \XX \to \YY$ a morphism of derived stacks. We write $\RRlHom_{\XX}(-,-)$ to denote the quasicoherent sheaf of morphisms between quasicoherent sheaves over a derived stack $\XX$, and $\RRlHom_f(-,-)$ for the composite $\RR f_\ast \RRlHom_{\XX}(-,-)$. In the particular case where $f\colon \ZZ\to\RSpec A$ is a morphism to an affine base, $\RRlHom_f(-,-)$ computes the dg-$A$-module of global sections of $\RRlHom_{\ZZ}(-,-)$. Via the Dold-Kan correspondence, such dg-$A$-module corresponds to a simplicial abelian group, hence to a homotopy type, and  we denote both its realisation as a dg-$A$-module and its underlying space as $\RHom_{\ZZ}(-,-)$. Finally, if $\XX$ is a derived stack, $\CE \in L_{\qcoh}(\XX)$ is a quasicoherent sheaf and $\RSpec A$ is an affine derived scheme, we denote by $\XX_A$ the product $\XX \times \RSpec A$ and by $\CE_A$ the pullback of $\CE$ along the projection $\XX_A \to \XX$. 
\end{conventions*}

\section{Derived geometry background}
\label{sec:background}

\subsection{Affine derived schemes}\label{sec: affine derived}
Let $\mathsf{cdga}_{\Bbbk}$ denote the ${\myInfty}$-category of commutative differential graded $\Bbbk$-algebras. In particular, for every $A \in \mathsf{cdga}_{\Bbbk}$, the $\Bbbk$-vector space $\HH^0(A)$ is an ordinary commutative $\Bbbk$-algebra, and each $\HH^j(A)$ is an $\HH^0(A)$-module.

Inside $\mathsf{cdga}_{\Bbbk}$ there is the full sub-${\myInfty}$-category $\cdga \subset \mathsf{cdga}_{\Bbbk}$ spanned by \textit{connective} commutative differential graded $\Bbbk$-algebras, i.e.~commutative differential graded $\Bbbk$-algebras whose cohomology vanishes in positive degrees.  Its opposite category 
\[
\dAff = \smash{\left(\cdga\right)^{\mathrm{op}}}
\]
is, by definition, the ${\myInfty}$-category of \emph{derived affine schemes}, and we denote the object  corresponding to a commutative differential graded $\Bbbk$-algebra $A$ by $\RSpec A$.

To any derived affine scheme $\RSpec A$ one can associate its \textit{underlying classical scheme} 
\[
t_0\RSpec A=\Spec \mathrm{H}^0(A).
\]
The operation $t_0$ is a functor $\dAff \to \mathsf{Aff}_{\Bbbk} = \mathsf{Alg}_{\Bbbk}^{\mathrm{op}}$ with target the standard category of affine schemes, and is left adjoint to a (fully faithful) functor $i \colon \mathsf{Aff}_{\Bbbk} \into \dAff$. We say that a derived affine scheme $\RSpec A$ is \emph{discrete} if $\RSpec A\simeq t_0\RSpec A$, where we view $t_0\RSpec A$ as an object of $\dAff$ via the functor $i$.

\begin{definition}[{\cite[Sec.~$2.2.2$]{hag2}}]
\label{def:morphismsderived}
Let $\varphi\colon\RSpec B\to\RSpec A$ be a morphism of derived affine schemes.
\begin{enumerate}
    \item [\mylabel{def-1}{(1)}] The morphism $\varphi$ is \textit{flat} if the morphism $t_0\varphi\colon t_0\RSpec B\to t_0\RSpec A$ is flat, and the canonical map
\[
\begin{tikzcd}[row sep=large,column sep=large]
\HH^i(A)\otimes_{\HH^0(A)}\HH^0(B)
\arrow[r]
&
\HH^i(B)
\end{tikzcd}
\]
is an isomorphism for every $i$.
    \item [\mylabel{def-2}{(2)}] The morphism $\varphi$ is \textit{\'etale}, \textit{smooth} or an \textit{open immersion} if it is flat and the morphism $t_0\varphi\colon t_0\RSpec B\to t_0\RSpec A$ of classical schemes is \'etale, smooth or an open immersion.
    \item [\mylabel{def-3}{(3)}] The morphism $\varphi$ is \emph{surjective} if $t_0\varphi\colon t_0\RSpec B\to t_0\RSpec A$ is surjective.
    \item [\mylabel{def-4}{(4)}] The morphism $\varphi$ is a \emph{covering} for the flat, \'etale, smooth or Zariski topology if it is a surjective morphism which is moreover flat, \'etale, smooth, or a disjoint union of open immersions respectively.
\end{enumerate}
\end{definition}

\begin{remark}
\label{remark:flatness}
In \cref{def:morphismsderived}\ref{def-1}, one could replace the second condition for flatness with the following, equivalent request: for any discrete dg-$A$-module $N$ (i.e., $\HH^j(N)\cong 0$ for $j\neq 0$), the derived tensor product $B\otimes_{A}^{\mathbf L}N$ is again discrete. We will use this criterion to detect flatness of coherent sheaves.
\end{remark}

\subsection{Derived stacks}
\label{sec:derivedstacks}
\Cref{def:morphismsderived}\ref{def-4} allows one to talk about topology and descent on $\cdga$. We shall mostly focus on \'etale descent in what follows.

Let $\Spaces$ be the ${\myInfty}$-category of spaces, i.e.~the $\myInfty$-category of topological spaces up to homotopy, cf.~\cite{htt}.
We denote by $\dSt$ the ${\myInfty}$-category of \emph{derived stacks} over $\Bbbk$, which is  the full sub-${\myInfty}$-category of the ${\myInfty}$-category of presheaves 
\[
\mathrm{Fun}\left(\cdga,\Spaces\right)
\]
spanned by sheaves for the étale topology on $\mathrm{cdga}^{\leqslant0}_{\Bbbk}$.

We say that a  derived stack $\XX$ is  a \emph{derived scheme} if it can be covered by Zariski open derived affine schemes $\UU \subset \XX$, and we denote their ${\myInfty}$-category by $\dSch\subset \dSt$.

\subsection{Truncation} \label{sec:truncation}
Denote by $\St$ the $\myInfty$-category of higher stacks \cite{Hirschowitz1998DescentePL}. The inclusion of $\myInfty$-categories $i\colon \St\hookrightarrow \dSt $ admits a left adjoint $t_0\colon \dSt\to \St$, which extends the truncation functor of derived affine schemes of  \Cref{sec: affine derived}. For a derived stack $\XX$, we call $t_0(\XX)$ the \emph{classical truncation} of $\XX$, and we say that   a derived stack  $\XX$ is a \textit{derived enhancement} of a classical (higher) stack $X$ if $t_0\XX\simeq X$.

\begin{warning}
Notice that, while the classical truncation of a derived scheme is a scheme in the ordinary sense, the classical truncation of a derived stack  is not a stack in the sense of ordinary algebraic geometry. This is because in general the image of a discrete commutative $\Bbbk$-algebra $A$ under the $\myInfty$-functor $t_0\XX$ can fail to be a $1$-groupoid. In fact, the ordinary $2$-category of classical algebraic stacks sits inside our $\myInfty$-category $\St$ as the full sub-$2$-category spanned by objects $F$ such that $F(A)$ is a $1$-truncated space for every discrete commutative algebra $A$.
\end{warning}

\subsection{Derived Artin stacks}
We say that a derived stack $\XX$ is \emph{geometric} if it is $m$-geometric, for some $m$, cf.~\cite[Def.~1.3.3.1]{hag2}. A \emph{derived Artin stack} is a geometric derived stack such that its truncation $t_0(\XX)$ is an $n$-Artin stack, for some $n$, cf.~\cite[Def.~4.4]{Toen-higher-and-derived}. We denote their $\myInfty$-category by $\dArt$. All in all, we have full inclusions 
\[\dAff \subset \dSch\subset \dArt \subset \dSt \]
of $\myInfty$-categories. 
  
For our purposes, the notion of geometricity for derived stacks is too restrictive. We say that a derived stack is \textit{locally geometric} if it has a cover by derived substacks which are geometric, cf.~\cite[Def.~2.17]{moduli-of-objects-dg}.

\subsection{Quasicoherent sheaves on derived stacks}
\label{sec:qcoh-derivestacks}

For a derived affine scheme $\RSpec A$, the  $\myInfty$-category of quasicoherent $\mathscr{O}_{\RSpec A}$-modules $L_{\qcoh}(\RSpec A)$ is by definition the $\myInfty$-category of dg-modules over $A$. It is a symmetric monoidal, presentable, stable and $\Bbbk$-linear $\myInfty$-category \cite[Sec.~7.1]{ha}. In particular, its homotopy category $\mathrm{Ho}(L_{\qcoh}(\RSpec A))$ is canonically triangulated and,  when $A$ is a discrete commutative $\Bbbk$-algebra, it agrees with the usual $\Bbbk$-linear triangulated  category $\derived(\Mod_A)$.

For a general derived stack $\boldit{X} \in \dSt$, the derived $\myInfty$-category of quasicoherent $\mathscr{O}_{\boldit{X}}$-modules $L_{\qcoh}(\boldit{X})$ is defined by étale descent as the limit of stable $\myInfty$-categories
\[
L_{\qcoh}(\boldit{X}) = \lim_{\substack{\RSpec A\to \boldit{X}}}L_{\qcoh}(\RSpec A).
\]
Intuitively, this means that a quasicoherent $\mathscr{O}_{\XX}$-module $\CM$ is a collection of quasicoherent $\mathscr{O}_{\RSpec A}$-modules $\CM_A$ for any derived affine cover, together with a homotopy-coherent system of gluing data. Again, $L_{\qcoh}(\boldit{X})$ is a symmetric monoidal, presentable, stable and $\Bbbk$-linear $\myInfty$-category.

\subsubsection{$t$-structures}
The $\myInfty$-category $L_{\qcoh}(\RSpec A)$ is equipped with an  explicit nondegenerate $t$-structure, where the $\myInfty$-category of connective (resp.~coconnective) objects $L_{\qcoh}(\RSpec A)^{\leqslant0}$ (resp.~$L_{\qcoh}(\RSpec A)^{\geqslant0}$) coincides with the $\myInfty$-category spanned by dg-modules whose cohomology vanishes in positive (resp.~negative) degrees. These two subcategories form a $t$-structure whose heart is naturally equivalent to the ordinary abelian category of $\HH^0(A)$-modules \cite[Prop.~$7.1.1.13$]{ha}.

For a general derived stack $\XX$, the $\myInfty$-category $L_{\qcoh}(\boldit{X})$ admits a canonical $t$-structure, whose $\myInfty$-category of (co)connective objects is  spanned by the quasicoherent $\mathscr{O}_{\boldit{X}}$-modules whose pullbacks along  morphisms $\RSpec A\to\boldit{X}$ are (co)connective as  dg-$A$-modules. In particular, given a flat morphism $f\colon\boldit{X}\to\boldit{Y}$ of derived Artin stacks, the pullback  $f^*\colon L_{\qcoh}(\boldit{Y})\to L_{\qcoh}(\boldit{X})$ is $t$-exact \cite[Prop.~$1.5.4$]{studyindag1}.

Given a derived stack $\XX$, 
the heart $L_{\qcoh}(\XX)^{\heartsuit}$  is always an ordinary abelian full subcategory of $L_{\qcoh}(\XX)$. Applying the ordinary machinery of truncations of objects in stable $\myInfty$-categories equipped with a $t$-structure (see \cite[Sec.~1.2.1]{ha}), for any integer $n\in\BZ$ we obtain  well-defined functors
\begin{equation*}
\begin{tikzcd}
\CH^n\defeq \tau_{\leqslant n}\circ\tau_{\geqslant n}\colon   L_{\qcoh}(\XX)\arrow[r]&L_{\qcoh}(\XX)^{\heartsuit}\subset L_{\qcoh}(\XX),
\end{tikzcd}
\end{equation*}
and we say that a quasicoherent $\mathscr{O}_{\XX}$-module $\CL$ is \emph{discrete} if $\CL\simeq\CH^0(\CL)$ in $L_{\qcoh}(\XX)$.

For a large class of nice derived stacks (e.g. derived schemes), one  has an equivalence
\[
L_{\qcoh}(\boldit{X})^{\heartsuit}\simeq\QCoh(t_0\boldit{X}),
\]
where $t_0\boldit{X}$ is its underlying  higher stack. 

\subsubsection{Flatness and surjectivity}
\label{sec:flatness-surj}
Fix  $A \in \cdga$ and a dg-module $M$ over $A$. Following \cite[Lemma $2.2.2.2$]{hag2}, we say that $M$ is \emph{flat over} $A$ if $\HH^0(M)$ is flat over $\HH^0(A)$ and the natural map
\[
\begin{tikzcd}
\HH^i(A) \otimes_{\HH^0(A)}\HH^0(M) \arrow[r]& \HH^i(M)
\end{tikzcd}
\]
is an isomorphism for ever $i$. Since dg-modules over $A$ correspond to quasicoherent sheaves on $\RSpec A$, this defines flatness for quasicoherent sheaves on affine derived schemes.

Let $f\colon \XX \to \YY$ be a morphism of derived stacks. We say that $\CF \in L_{\qcoh}(\XX)$ is flat over $\YY$ if for any discrete $\OO_{\YY}$-module $\CL$, one has that $\CF \otimes^{\mathbf{L}}f^\ast \CL$ is discrete as well. 

Let $\XX$ be a derived stack. A morphism $\CF \to \CG$ in $L_{\qcoh}(\XX)$ is \emph{surjective} if $\CH^0(\CF) \to \CH^0(\CG)$ is surjective. 

\begin{lemma}
\label{lemma:flat-is-open}
The following properties hold true.

\begin{itemize}
\item [\mylabel{openness}{\normalfont{(1)}}]
Let $f\colon Y\to S$ be a morphism locally of finite type between locally noetherian $\Bbbk$-schemes.

\begin{itemize}
\item [\mylabel{flat}{\normalfont{(i)}}] Let $\CF$ be a coherent $\OO_Y$-module. Then the locus
\[
\Set{y\in Y\mid \CF~\mathrm{is}~S\mbox{-}\mathrm{flat~at}~y} \subset Y
\]
is open.
\item [\mylabel{surjective}{\normalfont{(ii)}}] Assume $f$ is proper. Let $\vartheta\colon\CF\to\CG$ be a morphism of coherent $\OO_Y$-modules. Then the locus
\[
\Set{y\in Y\mid \vartheta_y~\mathrm{is}~\mathrm{surjective}} \subset Y
\]
is open.
\end{itemize}

\item [\mylabel{etale-local}{\normalfont{(2)}}] Both flatness and surjectivity, as defined in \Cref{sec:flatness-surj}, are local properties for the \'etale topology of derived stacks.
\end{itemize}
\end{lemma}

\begin{proof}
The proof of \ref{flat} is {\cite[\href{https://stacks.math.columbia.edu/tag/0398}{Tag 0398}]{stacks-project}}. The proof of \ref{surjective} goes as follows. Set $\CC = \coker \vartheta$. The scheme-theoretic support $Z=\Supp(\CC) \subset Y$ is closed and since $f$ is proper, $U=Y \setminus f(Z) \subset Y$ is open. But then $\CC|_{f^{-1}(U)} = 0$. Thus \ref{openness} is proved.

Let us prove \ref{etale-local}. Flatness of modules is a local condition for the flat topology by \cite[Prop.~2.8.4.2]{sag}, and thus for the \'etale topology as well. 

As for surjectivity, we need to show that if we have a derived stack $\XX$ along with a morphism $\CF \to \CG$ in $L_{\qcoh}(\XX)$, and an affine \'etale cover $\{f_i\colon \UU_i\to \XX\}$ such that the induced maps $f_i^*\psi$ are surjective, then $\psi$ is surjective too.

Notice that the surjectivity of a morphism $\psi\colon\CF \to \CG$ between  quasicoherent sheaves over a derived stack $\XX$ is equivalent to the fact that the induced morphism of discrete quasicoherent sheaves over $\XX$
\[
\begin{tikzcd}
\alpha\colon\CH^0(\CG)\arrow{r}&\CH^0(\Cone(\psi))   
\end{tikzcd}
\]
is the trivial map. By descent on the \'etale site, this can be checked on our given affine \'etale cover, i.e.~we only need to confirm that $f_i^\ast\alpha = 0$ for all $i$. Since each $f_i$ is in particular flat, its derived pullback is $t$-exact, therefore
\[
f_i^*\CH^0(\CG)\simeq\CH^0(f_i^*\CG)\quad\text{and}\quad f_i^*\CH^0(\Cone(\psi))\simeq \CH^0(f_i^*(\Cone(\psi)).
\]
Since the derived pullback preserves exact triangles, we have
\[f_i^*\Cone(\psi)\simeq \Cone(f_i^*(\psi)).\] 
It follows that the morphisms $f_i^*\alpha$ can be interpreted as the natural morphisms
\[
\begin{tikzcd}
\CH^0(f_i^*\CG)\arrow{r}&\CH^0\Cone(f_i^*(\psi)),
\end{tikzcd}
\]
which are all trivial since $f_i^*\psi$ are surjective by assumption. This proves that $\alpha=0$, confirming that surjectivity is local for the \'etale topology.
\end{proof}

\subsection{Cotangent complexes}
\label{sec:cotg-complex}
Let $\XX$ be a derived Artin stack. Then $\XX$ has a well-defined \emph{cotangent complex} $\BL_{\XX} \in L_{\qcoh}(\XX)$ of finite cohomological amplitude contained in $[-m,1]$ for some $m$, cf. \cite[Prop.~1.4.1.11]{hag2}. Moreover, by descent on the étale site, there is a global and well-defined cotangent complex $\BL_{\XX}$ on every locally geometric derived stack $\XX$, cf.~\cite[p.~431]{moduli-of-objects-dg}. Let $f\colon \XX\to \YY $ be a morphism of locally geometric derived stacks. Then there is a  natural map $f^*\BL_{\YY}\to \BL_{\XX} $, whose cone 
\[
\BL_f\defeq\Cone(f^*\BL_{\YY}\to \BL_{\XX} )
\]
we call the \emph{relative cotangent complex} attached to $f$. For a pair of morphisms $f\colon \XX \to \YY$ and $g\colon \YY \to \ZZ$ of locally geometric derived  stacks, the relative cotangent complexes fit into an exact triangle 
\[
\begin{tikzcd}
 f^\ast \BL_{g}
\arrow{r}
& \BL_{g\circ f}
\arrow{r}
&\BL_{f}   
\end{tikzcd}
\]
in $L_{\qcoh}(\XX)$, called the \emph{transitivity triangle}.

For a morphism of locally geometric derived stacks  $f\colon \XX\to \YY$,  we denote by $\BT_{f}$ the \emph{relative tangent complex}, i.e.~the $\mathscr{O}_{\XX}$-linear dual of $\BL_{f}$. The relative tangent complex satisfies the following functorial property (and a dual property can be formulated for the cotangent complex).
 
\begin{lemma}
\label{lemma:functoriality-tangents}
Let $\XX$, $\YY$ and $\ZZ$ be locally geometric derived stacks. If 
\[
\begin{tikzcd}[row sep=large,column sep=large]
\XX \times_{\ZZ} \YY \MySymb{dr} \arrow{r}{f}\arrow[swap]{d}{q} & \XX\arrow{d}{p} \\
\YY\arrow{r} & \ZZ
\end{tikzcd}
\]
is a cartesian diagram, then the canonical morphism between relative tangent complexes
\[
\begin{tikzcd}
\BT_{q}\arrow{r} & f^\ast \BT_p
\end{tikzcd}
\]
is an equivalence.
\end{lemma}
\begin{proof}
Since $\XX$, $\YY$ and $\ZZ$ are locally geometric, one can easily check via a local computation, using \cite[Prop.~7.3.3.7]{ha}, that the dual of our statement concerning cotangent complexes holds. Then the desired assertion will simply follow by $\mathscr{O}_{\XX\times_{\ZZ}\YY}$-linear duality, using \cite[Ex.~17.5.1.2]{sag}, once we know that the morphisms $p$ and $q$ in the diagram are \textit{infinitesimally cohesive} (in the sense of \cite[Def.~17.3.7.1]{sag}). However, the pullback $\XX \times_{\ZZ} \YY$ is again locally geometric \cite[Cor.~1.3.3.5]{hag2}, and using \cite[Prop.~17.3.8.4]{sag} it is easy to see that all morphisms of locally geometric derived stacks are infinitesimally cohesive.
\end{proof}

\subsection{Obstruction theories}
\label{sec:ob-theories}
Let $Y$ be a pure dimensional algebraic (i.e.~Artin) stack. Let $f\colon X \to Y$ be a morphism of algebraic stacks with unramified diagonal (such morphisms are called \emph{of Deligne--Mumford type}). In particular, $\HH^1(\BL_f)\cong 0$ and thus $\BL_f$ is an object in the derived category $\derived^{\leqslant 0}(\OO_{X_{\etale}})$. Following \cite[Def.~4.4]{BFinc}, an \emph{obstruction theory} relative to $f$ is a pair $(\BE,\phi)$, where 
\begin{itemize}
    \item [$\circ$] $\BE \in \derived^{\leqslant 0}(\OO_{X_{\etale}})$ is a complex with coherent cohomology in degrees $-1,0$ and vanishing cohomology in positive degrees, and
    \item [$\circ$] $\phi\colon \BE \to \BL_f$ is a morphism in $\derived^{\leqslant 0}(\OO_{X_{\etale}})$ such that $\HH^0(\phi)$ is an isomorphism and $\HH^{-1}(\phi)$ is surjective.
\end{itemize} 
We say that $(\BE,\phi)$ is \emph{perfect} if $\BE$ is of perfect amplitude contained in $[-1,0]$, cf.~\cite[Def.~5.1]{BFinc}.

One can define obstruction theories more generally, i.e.~for arbitrary morphisms $f\colon X \to Y$ of algebraic stacks, working with the lisse-\'etale site as defined in \cite{Olsson}. In this case, the relative cotangent complex $\BL_f$ is an element  of the derived category $\derived^{\leqslant 1}_{\coh}(\OO_{X_{\lissetale}})$, and an obstruction theory relative to $f$ is a pair $(\BE,\phi)$ where $\phi\colon \BE \to \BL_f$ is a morphism in $\derived^{\leqslant 1}_{\coh}(\OO_{X_{\lissetale}})$ such that $\HH^0(\phi)$ and $\HH^1(\phi)$ are isomorphisms and $\HH^{-1}(\phi)$ is surjective, cf.~\cite[Def.~3.1]{Poma-VCF-Artin}.

Derived algebraic geometry is an excellent provider of obstruction theories, in the sense of the following proposition.

\begin{prop}[{}]\label{prop:OT-truncated}
Let $\XX$ be a  locally geometric derived stack with truncation $X$. Let $\iota\colon X \into \XX$ be the canonical closed immersion. Then the natural morphism
\[
\begin{tikzcd}
\iota^\ast \BL_{\XX} \arrow{r} & \BL_X    
\end{tikzcd}
\]
is an obstruction theory on $X$.
\end{prop}
\begin{proof}
The proof follows directly from \cite[Prop.~1.2]{Schurg_Toen_Vezzosi_Determinants_of_perfect_complexes}.
\end{proof}

\section{The derived hyperquot scheme}
\label{sec:derived-nested-quot}

\subsection{The classical hyperquot scheme}
\label{sec:nested-quot-classical}
Throughout this section, we fix a polarised projective scheme $(X,\OO_X(1))$, a coherent sheaf $\CE\in \Coh(X)$ and a positive integer $l$. The \emph{hyperquot functor} attached to the triple $(X,\CE,l)$ is the functor
\begin{equation}
\label{classical-quot-functor}
\begin{tikzcd}
\mathsf{Quot}_X^{[l]}(\CE)\colon\mathsf{Sch}_{\Bbbk}^{\mathrm{op}} \arrow{r} & \Sets
\end{tikzcd}
\end{equation}
sending a scheme $B$ to the set of isomorphism classes of subsequent quotients
\begin{equation}\label{point-of-quot}
\begin{tikzcd}
\CE_B \arrow[two heads]{r} & \CT_1 \arrow[two heads]{r} & \cdots \arrow[two heads]{r} & \CT_l     
\end{tikzcd}
\end{equation}
in $\Coh(X \times B)$, where $\CT_i$ is $B$-flat for all $i=1,\ldots,l$ and $\CE_B$ denotes the pullback of $\CE$ along the projection $X\times B \to X$. Two such flags $\CE_B \onto \CT_\bullet$ and $\CE_B \onto \CT'_\bullet$ are isomorphic when $\ker(\CE_B \onto \CT_i) = \ker(\CE_B \onto \CT'_i)$ for all $i=1,\ldots,l$. The hyperquot functor is representable\footnote{This can be proved adapting the proof of \cite[Thm.~4.5.1]{Sernesi} or by an explicit induction on $l$ as in \cite[Sec.~2.A.1]{modulisheaves}.} by a $\Bbbk$-scheme $\Quot_X^{[l]}(\CE)$ locally of finite type, which we  call  the \textit{hyperquot scheme}. The scheme $\Quot_X^{[l]}(\CE)$ decomposes as an infinite disjoint union 
\[
\Quot_X^{[l]}(\CE)=\bigsqcup_{\boldit{p}}\Quot_X^{[l]}(\CE,\boldit{p})
\]
of open and closed projective subschemes indexed by $l$-tuples  $\boldit{p} = (p_1,\ldots,p_l)$ of polynomials $p_i \in \BQ[z]$. The $\boldit{p}$-th component $\Quot_X^{[l]}(\CE,\boldit{p})$ parametrises flags as in \eqref{point-of-quot} where $\CT_i$ has fibrewise Hilbert polynomial (with respect to the fixed polarisation $\OO_X(1)$) equal to $p_i$, for all $i=1,\ldots,l$. The smoothness of $\Quot_X^{[l]}(\CE,\boldit{p})$ was characterised in \cite{Lissite-quot} in the case where $ p_1, \dots, p_l$ are constant.

In \Cref{subsec:derived-quot-functor} we shall define a derived enhancement of the moduli functor $\mathsf{Quot}_X^{[l]}(\CE)$.

\subsection{Flags of perfect complexes} \label{subsec:derived-flags}
In this section we introduce the derived stacks of perfect complexes and flags thereof. Their geometricity properties, confirmed in \Cref{sec:geometricity-results}, will be used crucially in \Cref{subsec:derived-quot-functor}.

\subsubsection{The stack of perfect complexes}
Given $A \in \cdga$, the stable $\myInfty$-category $\mathsf{Perf}(A)$ can be defined as the full sub-$\myInfty$-category of $L_{\qcoh}(\RSpec A)$ spanned by the fully dualisable objects, see \cite[Prop.~2.7.28]{dagviii}. When $A$ is a classical commutative $\Bbbk$-algebra, this is just the sub-$\myInfty$-category spanned by objects which are isomorphic to bounded complexes of projective $A$-modules \cite[\href{https://stacks.math.columbia.edu/tag/07LT}{Tag 07LT}]{stacks-project}.
The functor
\[
\begin{tikzcd}[row sep=tiny]
\cdga\arrow{rr}{\mathsf{Perf}}
    &&
    \Cat_{\myInfty}\\
A\arrow[mapsto]{rr}&&\mathsf{Perf}(A)
\end{tikzcd}
\]
satisfies \'etale descent (see e.g.~\cite[Prop.~2.8.4.2]{sag}), so it defines a sheaf of $\myInfty$-categories for the étale topology on $\cdga$. In particular, it follows that for any derived stack $\YY$ the $\myInfty$-category $\mathsf{Perf}(\YY) \subset L_{\qcoh}(\YY)$, which is defined  as the full sub-$\myInfty$-category spanned by fully dualisable objects, satisfies the descent formula
\[
\mathsf{Perf}(\YY) \simeq \lim_{\RSpec A \to \YY}\mathsf{Perf}(A).
\]
Composing the functor $\mathsf{Perf}$ with the maximal subgroupoid functor
\[
\begin{tikzcd}
    (-)^{\simeq}\colon\Cat_{\myInfty}
    \arrow[r]
    &\Spaces,
\end{tikzcd}
\]
we obtain a derived stack that we denote by $\DPerf$. The fact that this composition satisfies descent as well follows from the fact that taking the maximal subgroupoid of a $\myInfty$-category is a right adjoint, hence preserves all limits.

For a derived stack $\XX$, we consider the ``relative'' version
\[
\DPerf(\XX)=\mathbf{Map}_{\dSt}(\XX,\DPerf),
\]
where $\mathbf{Map}_{\dSt}(-,-)$ denotes the internal hom-functor in the $\myInfty$-category of derived stacks. Informally, this stack parametrises perfect complexes over $\XX$. Moreover, when $\XX=X$ is a classical smooth and proper variety, $\DPerf(X)$ agrees with the locally geometric stack derived  stack $\mathscr{M}_{L_{\qcoh(X)}}$ of \emph{perfect} objects of $L_{\qcoh}(X)$, cf.~\cite{moduli-of-objects-dg}. 

\subsubsection{The stack of flags of perfect complexes}
Fix an integer $l \in \BZ_{\geqslant0}$. Let 
\[
[l]=\left\{0\to 1\to\cdots\to l\right\}
\]
be the standard $l$-simplex, seen as a category. The composition
\[
\begin{tikzcd}
\cdga\arrow{r}{\mathsf{Perf}} & \Cat_\myInfty \arrow{rr}{\Fun([l],-)} && \Cat_\myInfty \arrow{r}{(-)^{\simeq}} &\Spaces
\end{tikzcd}
\]
satisfies \'etale descent. Indeed, $\mathsf{Perf}$ satisfies \'etale descent, and the last two functors commute with arbitrary limits. Let us call $\mathsf{Perf}^{[l]}$ and $\DPerf^{[l]}$ the associated functor and derived stack, respectively. Again, for an arbitrary derived stack $\XX$, let
\[
\DPerf^{[l]}(\XX)=\mathbf{Map}_{\dSt}(\XX,\DPerf^{[l]})
\]
denote the relative version of $\DPerf^{[l]}$, parametrising flags
\begin{equation}\label{string-of-morphisms}
\begin{tikzcd}
\CE_\bullet={\{}
    \CE_0 \arrow{r} & 
    \CE_1 \arrow{r} &
    \cdots \arrow{r} &
    \CE_l{\}}
\end{tikzcd}
\end{equation}
of perfect complexes over $\XX$ of length $l+1$. 
For each $i=0,\ldots,l$ one has an \emph{evaluation morphism}
\[
\begin{tikzcd}[row sep=tiny]
\DPerf^{[l]}(\XX) \arrow{r}{\mathrm{ev}_i} & \DPerf(\XX) \\
\CE_\bullet \arrow[mapsto]{r} & \CE_i.
\end{tikzcd}
\]
This is an equivalence when $l=0$.

When $l>0$, we shall also make use of the natural morphisms
\begin{equation}\label{forgetful-morphism}
\begin{tikzcd}[row sep=tiny]
    \DPerf^{[l]}(\XX) \arrow{r}{\sigma_l} & \DPerf^{[l-1]}(\XX)\\
    \left\{\CE_0\to\cdots\to\CE_{l}\right\}\arrow[mapsto]{r} &  \left\{\CE_0\to\cdots\to\CE_{l-1}\right\}
\end{tikzcd}
\end{equation}
and
\begin{equation}
\label{forgetful-morphism-2}
\begin{tikzcd}[row sep=tiny]
\DPerf^{[l]}(\XX) \arrow{r}{\sigma_0} & \DPerf^{[l-1]}(\XX)\\
\left\{\CE_0\to\cdots\to\CE_{l}\right\}\arrow[mapsto]{r} & \left\{\CE_1\to\cdots\to\CE_{l}\right\}
\end{tikzcd}
\end{equation}
forgetting, respectively, the last arrow and the first arrow in a flag of length $l+1$.

\subsubsection{The stack of morphisms}
\label{sec:stack-MOR}
We will need in what follows another derived stack, defined as follows. Consider a pair of perfect complexes $\CF,\CG \in \mathsf{Perf}(\XX)$ over an arbitrary derived stack $\XX$. The pullback 
\begin{equation}
\label{def:morphismstack}
\begin{tikzcd}[row sep=large,column sep=large]
\mathbf{Mor}_{\XX}(\CF,\CG) \MySymb{dr}  \arrow{r}\arrow{d}
& \Spec \Bbbk \arrow{d}{\langle\CF,\CG\rangle}  \\
\DPerf^{[1]}(\XX) \arrow[swap]{r}{\langle s,t\rangle} 
& \DPerf(\XX)\times\DPerf(\XX)
\end{tikzcd}
\end{equation}
is a derived stack, which parametrises all morphisms $\CF \to \CG$ in $\mathsf{Perf}(\XX)$. In the diagram, we have written $s$ (as in source) for $\mathrm{ev}_0$ and $t$ (as in target) for $\mathrm{ev}_1$. Of course, replacing $\Spec \Bbbk$ with an arbitrary derived affine scheme $\RSpec A$ and letting $\CF$ and $\CG$ be arbitrary perfect complexes over $\XX_A$, one obtains a derived stack $\mathbf{Mor}_{\XX_A}(\CF,\CG)$ parametrising morphisms $\CF\to\CG$ in the category $\mathsf{Perf}(\XX_A)$.

\subsubsection{The stack of flags with fixed source}\label{sec:fixed source}
Let $\XX$ be a derived stack, $\CE \in \mathsf{Perf}(\XX)$ a perfect complex over $\XX$. One can consider the derived stack $\DPerf^{[l]}(\XX)_{\CE/}$ defined as the pullback 
\[
\begin{tikzcd}[column sep=large,row sep=large]
\DPerf^{[l]}(\XX)_{\CE/}\MySymb{dr}\arrow[hook]{d}  \arrow{r} & \Spec \Bbbk  \arrow[hook]{d}{\CE} \\
\DPerf^{[l]}(\XX) \arrow{r} & \DPerf(\XX)
\end{tikzcd}
\]
in the $\myInfty$-category $\dSt$ of derived stacks.

Informally, the stack $\DPerf^{[l]}(\XX)_{\CE/}$ parametrises flags as in \eqref{string-of-morphisms}, with the additional datum of an equivalence $\CE\simeq\CE_0$, together with equivalences of diagrams which make all the squares commute in a homotopy-coherent sense.

\subsection{Geometricity results}
\label{sec:geometricity-results}
The goal of this subsection is to prove that, given a derived stack $\XX$ such that $\DPerf(\XX)$ is locally geometric (see \Cref{thm:perf_locally_geometric} for sufficient conditions) and a perfect complex $\CE \in \mathsf{Perf}(\XX)$, the stacks
\begin{equation}\label{cacca-madonna}
\DPerf^{[l]}(\XX),\quad \DPerf^{[l]}(\XX)_{\CE/}
\end{equation}
are locally geometric for all $l\geqslant 0$ (cf.~\Cref{prop: Perf f loc geom} and \Cref{prop:Perf-fixed-E-locallygeometric} respectively).

Following \cite{integraltransforms}, we say that a derived stack $\XX$ is \emph{perfect} if $L_{\qcoh}(\XX)$ is a \emph{rigid symmetric monoidal} $\myInfty$-category, which means that the classes of compact and fully dualisable objects coincide. Moreover, we say that $L_{\qcoh}(\XX)$ is \emph{locally compact} if the $\Bbbk$-dg-module of morphisms between compact objects is compact.

\begin{theorem}[{\cite[Thm.~$4.4.1$]{Pandit2011ModuliPI}}]
\label{thm:perf_locally_geometric}
Let $\XX$ be a perfect derived stack such that $L_{\qcoh}(\XX)$ is compactly generated and locally compact. Then the derived stack $\DPerf(\XX)$ is locally geometric.
\end{theorem}

\begin{remark}
\label{rmk:perfect-stack}
For the reader's convenience, we recall some classes of perfect derived stacks to which \Cref{thm:perf_locally_geometric} applies.
\begin{itemize}
    \item [\mylabel{perf-1}{(1)}] Any quasiseparated and quasicompact classical scheme is a perfect derived stack \cite[Cor.~$2.3$]{Neeman2}. Any proper scheme satisfies the assumptions of \Cref{thm:perf_locally_geometric}  \cite[Prop.~2.5]{Neeman2}.
    \item [\mylabel{perf-2}{(2)}] Any quasiseparated derived scheme with affine diagonal is a perfect derived stack \cite[Prop.~$3.19$]{integraltransforms}. Any proper derived scheme satisfies the assumptions of \Cref{thm:perf_locally_geometric}.
    \item [\mylabel{perf-3}{(3)}] Products of perfect derived stacks and pullbacks of perfect derived stacks along affine morphisms are again perfect \cite[Prop.~$3.24$]{integraltransforms}. Finite products of proper perfect derived stacks satisfy the assumptions of \Cref{thm:perf_locally_geometric}. 
\end{itemize}
\end{remark}

\begin{prop}
\label{prop:Mor-locally-geometric}
Let $\XX$ be a derived stack such that $L_{\qcoh}(\XX)$ is locally compact. Fix a derived affine scheme $\RSpec A$ and two perfect complexes $\CF,\CG \in \mathsf{Perf}(\XX_A)$. Then the derived stack $\mathbf{Mor}_{\XX_A}(\CF,\CG)$ is geometric.
\end{prop}

\begin{proof}
The proof is  analogous to the one of \cite[Cor.~$3.19$]{moduli-of-objects-dg}, but some remarks are in order. We have that $\mathbf{Mor}_{\XX_A}(\CF,\CG)$ is equivalent to the derived stack described by the assignment
\[
\begin{tikzcd}[row sep=tiny]
\mathsf{cdga}^{\leqslant0}_A\arrow{r}&\Spaces\\
B\arrow[mapsto]{r}&\RHom_A(\RHom_{\XX_A}(\CF,\CG)),B).
\end{tikzcd}
\]
Since $L_{\qcoh}(\XX)$ is locally compact, we know that $L_{\qcoh}(\XX)\otimes_{L_{\qcoh}(\Spec \Bbbk)}L_{\qcoh}(\RSpec A)$ is locally compact over $A$ thanks to \cite[Lemma $4.2.15$]{Pandit2011ModuliPI}, and so is $L_{\qcoh}(\XX_A)$ because of the equivalence
\[
L_{\qcoh}(\XX_A)\simeq L_{\qcoh}(\XX)\otimes_{L_{\qcoh}(\Spec \Bbbk)}L_{\qcoh}(\RSpec A),
\]
see for example \cite[Cor.~$9.4.2.4$]{sag}. In particular, $\RHom_{\XX_A}(\CF,\CG)$ is a perfect $A$-dg-module, and so we can safely apply \cite[Sub-Lemma $3.9$]{moduli-of-objects-dg} just as in \cite[Cor.~$3.19$]{moduli-of-objects-dg}.
\end{proof}

For the rest of the section, we fix a derived stack $\XX$ satisfying the assumptions of \cref{thm:perf_locally_geometric}, so that in particular 
\begin{itemize}
    \item [\mylabel{cacca-1}{(i)}] $\DPerf(\XX)$ is locally geometric, and
    \item  [\mylabel{cacca-2}{(ii)}] \Cref{prop:Mor-locally-geometric} applies.
\end{itemize}

We can now state our desired geometricity results for derived stacks of flags of perfect complexes over $\XX$.
\begin{prop}\label{prop: Perf f loc geom}
Let $\XX$ be a derived stack satisfying the assumptions of \cref{thm:perf_locally_geometric}. Then $\DPerf^{[l]}(\XX)$ is locally geometric for every $l\geqslant 0$. 
\end{prop}
\begin{proof}
We argue by induction on $l$. For $l=0$, this is just the derived stack $\DPerf(\XX)$, whose local geometricity is precisely \cref{thm:perf_locally_geometric}. So assume that $\DPerf^{[l-1]}(\XX)$ is locally geometric. Consider the natural morphism of derived stacks
\begin{equation}\label{map-theta}
\begin{tikzcd}[row sep=tiny,column sep=large]
\DPerf^{[l]}(\XX)
\arrow{r}{\vartheta = \langle\sigma_l,\mathrm{ev}_l\rangle}
  &
  \DPerf^{[l-1]}(\XX)\times\DPerf(\XX)\\
   \left\{\CE_0\to\cdots\to\CE_{l}\right\} \arrow[mapsto]{r} & \left( \left\{\CE_0\to\cdots\to\CE_{l-1}\right\}, \,\CE_l\right)
\end{tikzcd}
\end{equation}
where $\sigma_l$ was defined in \eqref{forgetful-morphism}. The target is  locally geometric, since local geometricity is stable under homotopy pullbacks and therefore under finite products \cite[Cor.~$1.3.3.5$]{hag2}. Notice that we have a decomposition
 \begin{align*}
\DPerf(\XX)\simeq\bigcup_{a\leqslant b}\DPerf(\XX)^{[a,b]},     
\end{align*}
where $\DPerf(\XX)^{[a,b]}$ is the substack of $\DPerf(X)$ which parametrises perfect complexes over $\XX$ of fixed tor-amplitude $[a,b]$. This induces a decomposition 
\begin{align}\label{eqn: strata_l}
    \DPerf^{[l]}(\XX)\simeq\bigcup_{\substack{a_i\leqslant b_i\\
i=0,\ldots,l}}\DPerf^{[l]}(\XX)^{[a_i,b_i]},
\end{align}
where $\DPerf^{[l]}(\XX)^{[a_i,b_i]}$ denotes the pullback of derived stacks
\[
\begin{tikzcd}[row sep=large]
\DPerf^{[l]}(\XX)^{[a_i,b_i]} \MySymb{dr}  \arrow{r}\arrow[hook]{d} 
& \prod_{i=0}^l\DPerf(\XX)^{[a_i,b_i]} \arrow[hook]{d} \\
\DPerf^{[l]}(\XX) \arrow[swap]{r}{\langle \mathrm{ev}_i\rangle} 
& \prod_{i=0}^l\DPerf(\XX). 
\end{tikzcd}
\]
Explicitly, $\DPerf^{[l]}(\XX)^{[a_i,b_i]}$ is the substack of $\DPerf^{[l]}(\XX)$ which parametrises strings of morphisms of perfect complexes over $\XX$ of length $l+1$, as in \eqref{string-of-morphisms}, such that each $\CE_i$ has tor-amplitude $[a_i,b_i]$.
Consider the pullback diagram
\[
\begin{tikzcd}[row sep=large]
\DPerf^{[l]}(\XX)^{[a_i,b_i]}\MySymb{dr}  \arrow{r}\arrow[hook]{d} 
&\DPerf^{[l-1]}(\XX)^{[a_i,b_i]}\times\DPerf(\XX)^{[a_l,b_l]}\arrow[hook]{d} \\
\DPerf^{[l]}(\XX)\arrow[swap]{r}{\vartheta} & 
\DPerf^{[l-1]}(\XX)\times\DPerf(\XX)
\end{tikzcd}
\]
where $\vartheta$ was defined in \eqref{map-theta}.
We now show that the upper arrow  is $m$-representable for some $m$, i.e.~that given  $A\in\cdga$ and a map 
\begin{equation}\label{map-f}
\begin{tikzcd}
f\colon \RSpec A\arrow{r} & \DPerf^{[l-1]}(\XX)^{[a_i,b_i]}\times\DPerf(\XX)^{[a_l,b_l]},
\end{tikzcd}
\end{equation}
the pullback
\begin{equation}\label{crazy-pullback}
   f^\ast \DPerf^{[l]}(\XX)^{[a_i,b_i]} \defeq \DPerf^{[l]}(\XX)^{[a_i,b_i]}\times_{\DPerf^{[l-1]}(\XX)^{[a_i,b_i]}\times\DPerf(\XX)^{[a_l,b_l]}}\RSpec A 
\end{equation}
is a geometric derived stack. By the inductive hypothesis, this will imply the geometricity of $\DPerf^{[l]}(\XX)^{[a_i,b_i]}$, as explained in the proof of \cite[Prop.~$1.3.3.3(3)$]{hag2}, and hence our claim. 

So consider an $A$-valued point as in \eqref{map-f}, corresponding to a pair 
\begin{equation}\label{pair}
\left(\CE_{\bullet}\defeq\left\{\CE_0\to\cdots\to\CE_{l-1}\right\},\CE_l\right),
\end{equation}
where each $\CE_i$ is a perfect complex over $\XX_A$ of tor-amplitude $[a_i,b_i]$. Then the derived stack \eqref{crazy-pullback} parametrises all possible extensions of the pair \eqref{pair} to a flag
 \begin{equation}\label{flag}
\begin{tikzcd}
\CE_0 \arrow{r} &
    \CE_1 \arrow{r} & 
    \cdots \arrow{r} &
     \CE_{l-1} \arrow{r} &
    \CE_{l}.
\end{tikzcd}
\end{equation}
It is straightforward to see that the problem of extending a pair $(\CE_{\bullet},\CE_l)$ to a flag as in \eqref{flag} is equivalent to the problem of providing a morphism $\CE_{l-1}\to\CE_l$. In other words, specialising \Cref{prop:Mor-locally-geometric} to $\XX_A$, $\CF\defeq\CE_{l-1}$ and $\CG\defeq \CE_l$, we have that the canonical morphism
\[
\begin{tikzcd}
\mathbf{Mor}_{\XX_A}(\CE_{l-1},\CE_l) \arrow{r}&f^\ast \DPerf^{[l]}(\XX)^{[a_i,b_i]},
\end{tikzcd}
\]
obtained by gluing a morphism $\CE_{l-1}\to\CE_l$ to the flag $\CE_{\bullet}$, is an equivalence of derived stacks. Thus we obtain the geometricity of $f^\ast \DPerf^{[l]}(\XX)^{[a_i,b_i]}$ by an application of \Cref{prop:Mor-locally-geometric}.
\end{proof}

\begin{prop}\label{prop:Perf-fixed-E-locallygeometric}
Let $\XX$ be a derived stack satisfying the assumptions of \cref{thm:perf_locally_geometric}. Then $\DPerf^{[l]}(\XX)_{\CE/}$ is locally geometric for every $\CE \in \mathsf{Perf}(\XX)$ and every $l\geqslant 0$. 
\end{prop}

\begin{proof}
The case $l=0$ is trivial, so we assume $l>0$. The decomposition \eqref{eqn: strata_l} induces the  decomposition
\[
\DPerf^{[l]}(\XX)_{\CE/}\simeq\bigcup_{\substack{a_i\leqslant b_i\\
i=1,\ldots, l}}\DPerf^{[l]}(\XX)^{[a_i,b_i]}_{\CE/}
\]
where $\DPerf^{[l]}(\XX)^{[a_i,b_i]}_{\CE/}$ denotes the pullback of derived stacks
\[
\begin{tikzcd}[row sep=large,column sep=large]
\DPerf^{[l]}(\XX)^{[a_i,b_i]}_{\CE/}\arrow{r}\MySymb{dr}\arrow[hook]{d} &   \DPerf^{[l-1]}(\XX)^{[a_i,b_i]}\arrow[hook]{d}\\
\DPerf^{[l]}(\XX)_{\CE/}\arrow[swap]{r}{\sigma_{0,\CE}} & \DPerf^{[l-1]}(\XX)
\end{tikzcd}
\]
in which the bottom arrow $\sigma_{0,\CE}$ is the restriction of the map $\sigma_0$ (discarding the source $\CE$), defined in \eqref{forgetful-morphism-2}, to the closed substack $\DPerf^{[l]}(\XX)_{\CE/}\subset \DPerf^{[l]}(\XX)$. As in the proof of \Cref{prop: Perf f loc geom}, since $\DPerf^{[l-1]}(\XX)$ is locally geometric by \Cref{prop: Perf f loc geom}, to conclude it is sufficient to prove that the upper arrow is $m$-representable for some $m$, i.e.~that given a map $f\colon \RSpec A\to\DPerf^{[l-1]}(\XX)^{[a_i,b_i]}$, the pullback
\[
f^\ast \DPerf^{[l]}(\XX)^{[a_i,b_i]}_{\CE/} \defeq \DPerf^{[l]}(\XX)^{[a_i,b_i]}_{\CE/}\times_{\DPerf^{[l-1]}(\XX)^{[a_i,b_i]}}\RSpec A
\]
is $m$-geometric for some $m$. Such derived stack parametrises all possible ways of extending a sequence
\[
\begin{tikzcd}
    \CE_1 \arrow{r} & 
    \CE_2 \arrow{r} &
    \cdots \arrow{r} &
    \CE_{l}
\end{tikzcd}
\]
of perfect complexes over $X_A$, where each $\CE_i$ has tor amplitude $[a_i,b_i]$, to a sequence
\[
\begin{tikzcd}
\CE_A \arrow{r} &
    \CE_1 \arrow{r} & 
    \CE_2 \arrow{r} &
    \cdots \arrow{r} &
    \CE_{l},
\end{tikzcd}
\]
where $\CE_A$ denotes, as ever, the pullback of $\CE$ along $\XX_A \to X$. 
In other words, there is an equivalence of derived stacks
\[
\begin{tikzcd}
\mathbf{Mor}_{\XX_A}(\CE_A,\CE_1)\arrow{r}{\sim} & f^\ast \DPerf^{[l]}(\XX)^{[a_i,b_i]}_{\CE_A/},
\end{tikzcd}
\]
so we may conclude just as we did in \Cref{prop: Perf f loc geom}.
\end{proof}

\subsection{Derived hyperquot scheme}
\label{subsec:derived-quot-functor}
Let $X$ be a projective scheme, $\CE \in \Coh(X)$ a \emph{perfect} coherent sheaf (cf.~\Cref{sec:Amplifications}) and $l$ a positive integer. In this section we define a derived enhancement of the hyperquot functor \eqref{classical-quot-functor}. We first introduce the derived moduli functor $\RQuot_X^{[l]}(\CE)$, and we immediately confirm, in \Cref{lemma:RQuot-is-derived-stack}, that it is a derived stack and later in \Cref{RQuot-is-derived-scheme} that it is, in fact, a derived scheme.

\subsubsection{Definition and stackyness of \texorpdfstring{$\RQuot_X^{[l]}(\CE)$}{}}
Fix $A\in\cdga$. Consider the full subspace of
$\DPerf^{[l]}(X)_{\CE/}(A)$ spanned by those objects $\CE_A\to\CE_{1}\to\cdots\to \CE_{l}$ such that each $\CE_{i}$ is flat over $\RSpec A$ and every morphism is surjective. Since both conditions are stable under derived base change, this defines a functor
\begin{equation}\label{derived-quot-functor}
\begin{tikzcd}
\RQuot_X^{[l]}(\CE) \colon \cdga \arrow{r} & \Spaces
\end{tikzcd}
\end{equation}
which is a subfunctor of $\DPerf^{[l]}(X)_{\CE/}$. We call it the \emph{derived hyperquot functor} associated to $(X,\CE,l)$.

\begin{lemma}\label{lemma:RQuot-is-derived-stack}
The derived hyperquot functor $\RQuot_X^{[l]}(\CE)$ is a derived stack.
\end{lemma}

\begin{proof}
By definition, the derived moduli functor $\RQuot_X^{[l]}(\CE)$ maps any  $A\in\cdga$ to a union of connected components of $\DPerf(X)^{[l]}_{\CE/}(A)$. Since $A\mapsto \DPerf(X)^{[l]}_{\CE/}(A)$ is itself a derived stack, our claim follows from the fact that the conditions of surjectivity and flatness for quasicoherent sheaves on derived stacks are local for the \'etale topology, which is precisely the content of \Cref{lemma:flat-is-open}\ref{etale-local}. 
\end{proof}

\subsubsection{Geometricity of  \texorpdfstring{$\RQuot_X^{[l]}(\CE)$}{}}

The goal of this subsection is to prove that $\RQuot_X^{[l]}(\CE)$ is locally geometric (cf.~\Cref{cor:Rquot-is-locally-geom-stack} below).

We recall a foundational result which relates open derived substacks of a derived stack $\XX$ and open substacks of its classical truncation.
\begin{prop}[{\cite[Prop.~2.1]{Schurg_Toen_Vezzosi_Determinants_of_perfect_complexes}}]
\label{prop:opensubstacks}    
Let $\XX$ be a derived stack and let $X = t_0\XX$ be classical truncation. There is a bijective correspondences of equivalence classes
\[
\begin{tikzcd}
\begin{Bmatrix}
    \text{Zariski open substacks of }X
\end{Bmatrix}\arrow{r}{\phi_{\XX}}& \begin{Bmatrix}
    \text{Zariski open derived substacks of }\XX
\end{Bmatrix}. 
\end{tikzcd}
\]
Moreover, there is a pullback diagram of derived stacks
\[
\begin{tikzcd}[column sep=large,row sep=large]
U \MySymb{dr} \arrow[hook]{r}\arrow[hook]{d} & X \arrow[hook]{d} \\
\phi_{\XX}(U)\arrow[hook]{r} & \XX
\end{tikzcd}
\]
for any Zariski open substack $U\into X$.

\end{prop}
Explicitly, the image of an open substack $U\into X$ under the bijection $\phi_{\XX}$ of \cref{prop:opensubstacks} is the derived stack $\UU$ given by
\[
\begin{tikzcd}[row sep=tiny]
\cdga\arrow{r}{\UU}&\Spaces\\
A\arrow[mapsto]{r}&\XX(A)\times_{X(\HH^0(A))}U(\HH^0(A)).
\end{tikzcd}
\]
\begin{prop}
\label{prop:openimmersion}
There is an open immersion 
\[
\begin{tikzcd}
j_{\CE}\colon \RQuot_X^{[l]}(\CE) \arrow[hook]{r} & \DPerf^{[l]}(X)_{\CE/}.
\end{tikzcd}
\]
\end{prop}
\begin{proof}
By \cref{prop:opensubstacks}, we need to prove that $t_0\RQuot^{[l]}_X(\CE)$ is an open substack of $t_0\DPerf^{[l]}(X)_{\CE/}$, and that  $\phi_{\DPerf^{[l]}(X)_{\CE/}} (t_0\RQuot^{[l]}_X(\CE))\cong \RQuot^{[l]}_X(\CE)$.

The first statement can be checked on an open affine cover of $t_0\DPerf^{[l]}(X)_{\CE/}$. Let therefore  $\Spec A$ be a classical affine scheme and consider an $A$-valued point $\Spec A \to t_0\DPerf^{[l]}(X)_{\CE/}$ corresponding to a sequence 
\[
\begin{tikzcd}
    \CE \arrow{r} & 
    \CE_1 \arrow{r} &
    \cdots \arrow{r} &
    \CE_{l},
\end{tikzcd}
\]
such that each map is surjective after taking the $\CH^0$-sheaves and such that each $\CE_i$ is flat. Then, the pullback
\[
\begin{tikzcd}
\Spec A \times_{t_0\DPerf^{[l]}(X)_{\CE/}}t_0\RQuot^{[l]}_X(\CE)\arrow{r} & \Spec A
\end{tikzcd}
\]
corresponds to an open immersion of classical schemes, since the flat and surjective locus is open by \Cref{lemma:flat-is-open}\ref{openness}.

For the second statement, let $\RSpec A$ be 
a derived affine scheme. We need to show that  the natural morphism of spaces
\[
\begin{tikzcd}
\varphi_A\colon\RQuotlX{l}(A)\arrow{r}&\DPerflX{l}_{\CE/}(A)\times_{t_0\DPerflX{l}(\HH^0(A))}t_0\RQuotlX{l}(\HH^0(A))
\end{tikzcd}
\]
induced by the universal property of pullbacks is an equivalence.

Since the inclusions of both source and target of $\varphi_A$ inside $\DPerflX{l}(A)$ are union of full connected components (i.e., all morphisms between $0$-simplices are preserved), it is sufficient to prove that $\varphi_A$ is a bijection at the level of $0$-simplices. The $0$-simplices in the target of $\varphi_A$ correspond to the datum of a sequence of morphisms
\[
\begin{tikzcd}
    \CE_0\defeq\CE_A\arrow{r}
& \CE_1\arrow{r}&\cdots\arrow{r}
& \CE_l
\end{tikzcd}
\]
of perfect complexes over $X_A = X\times \RSpec A$, such that 
\[
\widetilde{\CE}_i\defeq\CE_i\otimes^{\mathbf L}_{\mathscr{O}_{X_A}}\mathscr{O}_{X_{\HH^0(A)}}
\]is flat over $\Spec \HH^0(A)$ and the maps
\[
\begin{tikzcd}
    \widetilde{\CE}_0\arrow{r}
& \widetilde{\CE}_1 \arrow{r}&\cdots\arrow{r}
& \widetilde{\CE}_l
\end{tikzcd}
\]
are all surjective. On the other hand, the $0$-simplices in  the source of $\varphi_A$ correspond to the datum of a sequence of morphisms
\[
\begin{tikzcd}
    \CE_A\arrow{r}
& \CE_1\arrow{r}&\cdots\arrow{r}
& \CE_l
\end{tikzcd}
\]
of perfect complexes over $X_A$ such that each $\CE_i$ is flat over $\RSpec A$ and each morphism is surjective. In particular, it follows immediately that the map $\varphi_A$ is injective at the level of $0$-simplices.

To prove surjectivity, let $\CE_A\to\CE_1\to\cdots\to\CE_l$ correspond to a $0$-simplex in the target of $\varphi_A$. Let $\CL$ be a discrete quasicoherent sheaf over $\RSpec A$ and $\pi_A\colon X_A\to\RSpec A$ be the canonical projection. 
Then, since each $\widetilde{\CE}_i$ is flat over $\Spec \HH^0(A)$, we have that
\begin{align*}
 \CE_i\otimes^{\mathbf L}_{\mathscr{O}_{X_A}}\pi^*_A\CL&\simeq\widetilde{\CE}_i\otimes^{\mathbf L}_{\mathscr{O}_{X_{\HH^0(A)}}}\pi^*_{\HH^0(A)}\CH^0(\CL)
\end{align*}
is a discrete object as well, proving that each $\CE_i$ is flat over $\RSpec A$. 
 
Finally,  each morphism $\widetilde{\CE}_i\to\widetilde{\CE}_{i+1}$ is assumed to be surjective, which implies that the induced morphisms $ \CH^0(\CE_i)\to\CH^0(\CE_{i+1})$ are surjective as well, by which we conclude that $\varphi_A$ is surjective.
\end{proof}

\begin{corollary}
\label{cor:Rquot-is-locally-geom-stack}
The derived moduli stack $\RQuot_X^{[l]}(\CE)$ is locally geometric.
\end{corollary}

\begin{proof}
This follows at once from Propositions \ref{prop:Perf-fixed-E-locallygeometric}, \ref{prop:openimmersion} and from the fact that an open substack of a locally geometric stack is again locally geometric.
\end{proof}

\subsubsection{Representability of \texorpdfstring{$\RQuot_X^{[l]}(\CE)$}{}}
The goal of this section is to prove that $\RQuot_X^{[l]}(\CE)$ is a derived enhancement of the classical hyperquot scheme $ \Quot_X^{[l]}(\CE)$ and, in particular, that it is represented by a derived scheme.

The next result follows from a generalisation of the strategy employed by Adhikari's in \cite[Sec.~3.2]{DerivedQuot} for the case $l=1$.

\begin{theorem}
\label{thm:truncation}
There is an equivalence of higher stacks
\[
t_0(\RQuot_X^{[l]}(\CE)) \simeq \Quot_X^{[l]}(\CE).
\]
In particular, $t_0(\RQuot_X^{[l]}(\CE))$ is a classical scheme.
\end{theorem}
\begin{proof}
We split the proof in three steps.

\emph{Step I.} There is  a natural map
\[
\begin{tikzcd}
\alpha\colon \Quot^{[l]}_X(\CE)\arrow{r} & t_0\RQuotlX{l}.
\end{tikzcd}
\]
Indeed, for any commutative algebra $A$, we have a 
natural map of spaces
\[
\begin{tikzcd}
\Quot^{[l]}_X(\CE)(A)\arrow{r} & t_0\RQuotlX{l}(A),
\end{tikzcd}
\]
since a sequence of 
surjective morphisms of discrete quasicoherent sheaves on $X_A$ which are flat over $\Spec A$
\[
\begin{tikzcd}
    \CE_A\arrow[two heads]{r}&\CE_1\arrow[two heads]{r}&\cdots\arrow[two heads]{r}&\CE_l
\end{tikzcd}
\]
is in particular a sequence of surjective morphisms of quasicoherent sheaves on $X_A$ which are flat over $\Spec A$ in the derived sense.

\emph{Step II.}
We claim that for any commutative algebra $A$,  the morphism $\alpha$ induces a bijection on the sets of connected components of the $A$-valued points.

For injectivity, assume that we are given two elements $\mathscr{E}, \mathscr{F}\in  \Quot^{[l]}_X(\CE)(A)$ which are equivalent in $t_0\RQuotlX{l}(A)$. This means that there exists a quasi-isomorphism $\mathscr{E}\to \mathscr{F}$, commutative up to coherent homotopy, which induces an isomorphism of complexes after applying $\CH^0$.

For surjectivity, notice that an element $\mathscr{E}\in t_0\RQuotlX{l}(A)$ is flat over $\Spec A$, and therefore discrete, which implies that it is the image of  $\CH^0(\mathscr{E})$.

\emph{Step III.} We claim that for any commutative algebra $A$ there is an equivalence of spaces
\[
\alpha\colon \Quot^{[l]}_X(\CE)(A)\simeq t_0\RQuotlX{l}(A).
\]
By the representability of $ \Quot^{[l]}_X(\CE)$, we have that $ \Quot^{[l]}_X(\CE)(A)$ is a set, therefore the claim follows once we prove that $\RQuotlX{l}(A)$ has trivial  higher homotopies. In other words, we need to prove that for all $\mathscr{E}\in \RQuotlX{l}(A)$, there are vanishings of homotopy groups $\pi_i(\RQuotlX{l}(A), \mathscr{E})=0$, for all $i>0$. Set 
\[
\mathscr{E}\defeq
{\{}
\CE_A\xrightarrow{~~~~~}\CE_1\xrightarrow{~~~~~}\cdots\xrightarrow{~~~~~}\CE_l
{\}},
\]
where each $\CE_i$ is a genuine coherent sheaf on $X_A$. Since 
\[
\RQuotlX{l}(A)\subset \DPerflX{l}_{\CE/}(A)
\]
is a union of connected components, we can compute its homotopy groups from $\DPerflX{l}_{\CE/}(A) $. Consider the natural map of spaces
\[
\begin{tikzcd}
\mathsf{Perf}^{[l]}(X_A)^{\simeq}_{\CE_A/}\arrow{r}&\mathsf{Perf}^{[l-1]}(X_A)^{\simeq}
\end{tikzcd}
\]
forgetting $\CE_A$, as in \eqref{forgetful-morphism-2}. This is naturally a left fibration \cite[Cor.~$2.1.2.2$]{htt}, hence in particular a Kan fibration \cite[Lemma $2.1.3.3$]{htt}; therefore, we can consider the homotopy fibre at the point $\mathscr{F}=\{\CE_1\to\cdots\to\CE_l\}$, namely
\[
\begin{tikzcd}
\mathrm{fib}\,
\Bigl(\mathsf{Perf}^{[l]}(X_A)^{\simeq}_{\CE_A/}\arrow{r}&\mathsf{Perf}^{[l-1]}(X_A)^{\simeq}\Bigr)
\simeq \RHom_{X_A}(\CE_A,\CE_1).
\end{tikzcd}\]
Applying Serre's long exact sequence of homotopy groups for Kan fibrations we obtain
\begin{multline*}
\cdots\xrightarrow{~~~~~~}
\pi_n(\RHom_{X_A}(\CE_A,\CE_1),\mathscr{E})\xrightarrow{~~~~~~}
\pi_n(\mathsf{Perf}^{[l]}(X_A)^{\simeq}_{\CE_A/},\mathscr{E})\\
\xrightarrow{~~~~~~}
\pi_n(\mathsf{Perf}^{[l-1]}(X_A)^{\simeq},\mathscr{F})
\xrightarrow{~~~~~~}\pi_{n-1}(\RHom_{X_A}(\CE_A,\CE_1),\mathscr{E})\xrightarrow{~~~~~~}\cdots
\end{multline*}
so we are only left to compute  homotopy groups of $\RHom_{X_A}(\CE_A, \CE_1) $ and $\mathsf{Perf}^{[l-1]}(X_A)^{\simeq}$. Such homotopy groups satisfy
\begin{align*}
\pi_i(\RHom_{X_A}(\CE_A,\CE_1),\mathscr{E})
&\cong\Ext^{-i}_{X_A}(\CE_A,\CE_1), & i \geqslant 0\\
\pi_i(\mathsf{Perf}^{[l-1]}(X_A)^{\simeq},\mathscr{F})
&\cong\Ext^{1-i}_{\mathsf{Perf}^{[l-1]}(X_A)}(\mathscr{F},\mathscr{F}), & i > 0.
\end{align*}
Notice that $\CE_A$ and $\CE_1$ are  discrete quasicoherent sheaves defined over the classical scheme $X_A$  flat over $\Spec A$, so 
\begin{align}\label{eqn: vanishing ext}
\Ext^{-i}_{X_A}(\CE_A,\CE_1)\cong0, \qquad i>0.    
\end{align}
On the other hand, we claim that 
\[
\Ext^{1-i}_{\mathsf{Perf}^{[l-1]}(X_A)}(\mathscr{F},\mathscr{F})\cong 0, \qquad i>1.
\]
Consider the evaluation maps $ \Ext^{1-i}_{\mathsf{Perf}^{[l-1]}(X_A)}(\mathscr{F},\mathscr{F})\to \Ext^{1-i}_{X_A}(\CE_j,\CE_j)$ for each $j=1, \dots, l-1$. We have that $\Ext^{1-i}_{\mathsf{Perf}^{[l-1]}(X_A)}(\mathscr{F},\mathscr{F}) $ vanishes if its image via each evaluation map vanishes, which follows once more from \eqref{eqn: vanishing ext}.

Chasing down Serre's long exact sequence, we obtain
\[
\pi_i(\mathsf{Perf}^{[l]}(X_A)^{\simeq}_{\CE_A/},\mathscr{E})\cong0, \qquad i>1,
\]
and an exact sequence
\[
0 \xrightarrow{~~~~~}
\pi_1(\mathsf{Perf}^{[l]}(X_A)^{\simeq}_{\CE_A/},\mathscr{E})\xrightarrow{~~~~~} 
\pi_1(\mathsf{Perf}^{[l-1]}(X_A)^{\simeq},\mathscr{F}) \xrightarrow{~~~~~}
\pi_0(\RHom_{X_A}(\CE_A,\CE_1),\mathscr{E}).
\]
The last  map can be described as
\[
\begin{tikzcd}[row sep=tiny]
\Hom_{\mathsf{Perf}^{[l-1]}(X_A)}(\mathscr{F},\mathscr{F})\arrow{r}{q_{1}^\ast}&\Hom_{X_A}(\CE_A,\CE_1)\\
\varphi=(\varphi_1,\ldots,\varphi_l)\arrow[maps to]{r} &\varphi_1\circ q_1,
\end{tikzcd}
\]
where $q_1\colon\CE_A\to\CE_1$ denotes the first surjection and where $\varphi_i \colon \CE_i \to \CE_{i}$ are the maps forming the given map $\SF \to \SF$. Notice that the required vanishing follows from the injectivity of $q_{1}^\ast$, which in turn follows from the surjectivity of $q_1$ and of each map $q_i\colon \CE_{i-1}\to \CE_i$, for $i=2, \dots, l$ via a simple inductive argument on $l$.
\end{proof}

\begin{corollary}
\label{RQuot-is-derived-scheme}
The derived stack $\RQuot_X^{[l]}(\CE)$ is representable by a derived scheme.
\end{corollary}
\begin{proof}
By \cref{cor:Rquot-is-locally-geom-stack},   $\RQuotlX{l}$ is a locally geometric derived stack and by \cref{thm:truncation}  its truncation is equivalent to the scheme  $\Quot^{[l]}_X(\CE)$. Therefore by applying \cite[Cor.~$1.6.7.4$]{sag} we deduce that $\RQuotlX{l}$ is a derived scheme.
\end{proof}

\section{Proof of Theorems \ref{MAIN-THEOREM-TG-CPLX} and \ref{MAIN-THEOREM-OBS-TH}}\label{sec:proof-of-main-thm}

\subsection{Tangent complex of source-and-target map}\label{sec:source-target-tangent}
Let $\XX$ be a derived stack satisfying the assumptions of \cref{thm:perf_locally_geometric}. In this section we provide an explicit description of the relative tangent complex $\BT_{\langle s,t\rangle}$ attached to the source-and-target map
\[
\begin{tikzcd}
\DPerf^{[1]}(\XX) \arrow{rr}{\langle s,t\rangle} && \DPerf(\XX)^{\times 2}  
\end{tikzcd}
\]
introduced in greater generality in \Cref{sec:stack-MOR}. This will be used crucially in \Cref{sec:cotangent-complex-quot}, starting from Diagram \ref{map-gamma}. To fix notation, we denote by
\[
\begin{tikzcd}
    \SF \arrow{r}{f^{\uni}} & \SG
\end{tikzcd}
\]
the universal map living over $\DPerf^{[1]}(\XX) \times \XX$.

\begin{prop}\label{prop:TANGENT-source-and-target}
Let $\XX$ be a derived stack satisfying the assumptions of \cref{thm:perf_locally_geometric}, e.g.~a proper classical scheme. There is an equivalence
\[
\begin{tikzcd}
\BT_{\langle s,t\rangle} \arrow{r}{\sim} & \RRlHom_\tau(\SF,\SG),
\end{tikzcd}
\]
where $\tau\colon \DPerf^{[1]}(\XX) \times \XX \to \DPerf^{[1]}(\XX)$ is the projection.
\end{prop}

\begin{proof}
Since both $\DPerf(\XX)$ and $\DPerf^{[1]}(\XX)$ are locally geometric derived stacks (cf.~\Cref{prop: Perf f loc geom}), they both admit global (co)tangent complexes and so does the map $\langle s,t\rangle$. Cotangent complexes are fully determined by local data in the sense of \cite[Rmk.~17.2.4.3]{sag}. Applying this to the cotangent complex of $\DPerf^{[1]}(\XX)$, we obtain the following description. For any $A \in \cdga$ and for any $A$-valued point $x\colon \RSpec A \to \DPerf^{[1]}(\XX)$, corresponding to a map $f\colon \CF\to\CG$ in $\mathsf{Perf}(\XX_A)$, the pullback diagram
\begin{equation*}
\begin{tikzcd}[row sep=large,column sep=large]
\RSpec A \arrow{dr}
\arrow[bend left=20]{drr}[description]{\id}
\arrow[bend right=20]{ddr}[description]{x} & & \\
& \mathbf{Mor}_{\XX_A}(\CF,\CG) \MySymb{dr} \arrow{r}\arrow{d}
& \RSpec A \arrow{d}{\langle\CF,\CG\rangle}  \\
& \DPerf^{[1]}(\XX) \arrow[swap]{r}{\langle s,t\rangle} 
& \DPerf(\XX)^{\times 2}
\end{tikzcd}
\end{equation*}
induces, after pulling back along $x$ the transitivity triangle attached to $\langle s,t\rangle$, an exact triangle 
\[
\RHom_{\XX_A}(\CF,\CF)^\vee[-1] \oplus \RHom_{\XX_A}(\CG,\CG)^\vee[-1] \xrightarrow{~~~~~} x^\ast \BL_{\DPerf^{[1]}(\XX)} \xrightarrow{~~~~~} \RHom_{\XX_A}(\CF,\CG)^\vee
\]
in $L_{\qcoh}(\RSpec A)$. We deduce that
\[
\begin{tikzcd}
x^\ast\BL_{\DPerf^{[1]}(\XX)} \simeq 
\Cone \biggl(\RHom_{\XX_A}(\CF,\CG)^\vee \arrow{r}{w} &  \RHom_{\XX_A}(\CF,\CF)^\vee \oplus \RHom_{\XX_A}(\CG,\CG)^\vee\biggr)[-1]
\end{tikzcd}
\]
where $w$ is dual of the map
\[
(\alpha\colon\CF \to \CF,\beta\colon \CG \to \CG) \mapsto f\circ \alpha - \beta \circ f.
\]
Therefore, globalising, we obtain 
\begin{equation}
\label{cotangent-Perf^1}
\begin{tikzcd}  
\BL_{\DPerf^{[1]}(\XX)}\simeq \Cone \biggl(\RRlHom_\tau(\SF,\SG)^\vee \arrow{r} {w^{\uni}} & \RRlHom_\tau(\SF,\SF)^\vee\oplus \RRlHom_\tau(\SG,\SG)^\vee \biggr)[-1],
\end{tikzcd}
\end{equation}
where $w^{\uni}$ is the dual of the map
\[
(\alpha\colon\SF \to \SF,\beta\colon \SG \to \SG) \mapsto f^{\uni}\circ \alpha - \beta \circ f^{\uni}.
\]
On the other hand, we know that $\DPerf(\XX)^{\times 2}$ admits a global cotangent complex, and using the same strategy of \cite[Cor.~3.29]{moduli-of-objects-dg} it can be described as
\[
\BL_{\DPerf(\XX)^{\times 2}}\simeq\RRlHom_{q}(\SU_1,\SU_1)^\vee[-1]\oplus \RRlHom_{q}(\SU_2,\SU_2)^\vee[-1],
\]
where $q\colon \DPerf(\XX)\times \XX\to \DPerf(\XX)$ is the projection and $\SU_i$ is the universal perfect complex over the $i$-th copy $\DPerf(\XX)\times \XX$ for $i=1,2$. Moreover, by the very definition of the map $\langle s,t\rangle$, we have 
\begin{align*}
(s \times \id_{\XX})^\ast \SU_1 &= \SF,\\
(t \times \id_{\XX})^\ast \SU_2 &= \SG.    
\end{align*}
This yields
\begin{equation}\label{pb-st}
\langle s,t\rangle^\ast \BL_{\DPerf(\XX)^{\times 2}} \simeq \RRlHom_\tau (\SF,\SF)^\vee[-1]
\oplus \RRlHom_\tau (\SG,\SG)^\vee[-1].
\end{equation}
Combining \eqref{cotangent-Perf^1} and \eqref{pb-st} with the transitivity triangle of cotangent complexes
\[
\begin{tikzcd}
\BL_{\langle s,t\rangle}[-1] \arrow{r} & \langle s,t\rangle^\ast \BL_{\DPerf(\XX)^{\times 2}}\arrow{r} &  \BL_{\DPerf^{[1]}(\XX)} \arrow{r} & \BL_{\langle s,t\rangle}
\end{tikzcd}
\]
and shifting back, we see that $\BL_{\langle s,t\rangle} \simeq \RRlHom_\tau(\SF,\SG)^\vee$. Dualising, we obtain the sought after result.
\end{proof}
We remark that a similar computation of relative tangent complexes for the moduli stack of extensions is carried on in \cite[Corollary 3.7]{Porta_Sala_Hall_algebras}.
\subsection{Tangent complex of \texorpdfstring{$\RQuotlX{l}$}{}}\label{sec:cotangent-complex-quot}
Let $X$ be a projective scheme, $\CE \in \Coh(X)$ a perfect coherent sheaf over $X$, and $l$ a positive integer. Form the derived hyperquot scheme
\[
\QQ \defeq \RQuotlX{l}
\]
and let 
\begin{equation}\label{universal flag}
\begin{tikzcd}
\SE_0^{\uni} = \CE_{\QQ}\arrow[two heads]{r}{\psi_1} & \SE_1^{\uni} \arrow[two heads]{r}{\psi_2} & \SE_2^{\uni} \arrow[two heads]{r}{\psi_3} & \cdots \arrow[two heads]{r}{\psi_l} & \SE_l^{\uni}    
\end{tikzcd} 
\end{equation}
be the universal flag living over $\QQ\times X$. Consider the open substacks 
\begin{align*}
\DPerf(X)_\fl &\,\,\subset\,\,\DPerf(X) \\
\DPerf^{[l]}(X)_\fl &\,\,\subset\,\, \DPerf^{[l]}(X) \\
\DPerf^{[l]}(X)_{\CE/,\fl} &\,\,\subset \,\,\DPerf^{[l]}(X)_{\CE/}
\end{align*}
parametrising \textit{flat} perfect complexes or flags thereof: the fact that these are indeed open substacks is contained in the proof of \cref{prop:openimmersion}.
Next, form the $l$-fold products
\begin{align*}
\PP_l & \defeq \DPerf(X)_\fl^{\times l} = \PP_1 \times \cdots \times \PP_1 \\
\PP_{[l]} & \defeq \DPerf^{[1]}(X)_\fl^{\times l} = \PP_{[1]} \times \cdots \times \PP_{[1]}.
\end{align*}
Our key focus will be on the morphism
\[
\begin{tikzcd}[row sep=tiny]
\QQ\arrow{r}{\hh} & \PP_l \\
{\{}\CE \to \CE_1 \to \cdots \to \CE_l{\}} \arrow[mapsto]{r} & (\CE_1,\ldots,\CE_l).
\end{tikzcd}
\]
To fix notation, form the cartesian diagrams
\begin{equation}\label{double-fibre-diagram}
\begin{tikzcd}[row sep=large,column sep=large]
\QQ\times X\MySymb{dr}\arrow{r}{\hh \times \id_X}\arrow{d}{\boldit{\pi}} & \PP_l \times X \MySymb{dr}\arrow{r}{\pr_i \times \id_X}\arrow{d}{p} & \PP_1 \times X\arrow{d}{q} \\
\QQ\arrow[swap]{r}{\hh} & \PP_l\arrow[swap]{r}{\pr_i} & \PP_1
\end{tikzcd}
\end{equation}
where $\pr_i$ denotes the $i$-th projection, for $i=1,\ldots,l$. Let $\SU_i$ be the universal complex on the $i$-th copy $\PP_1\times X$. Then its tangent complex is
\[
\BT_{\PP_l} \simeq \underset{1\leqslant i \leqslant l}{\Boxplus}\RRlHom_q(\SU_i,\SU_i)[1],
\]
and therefore we have
\begin{equation}\label{chain-pullbacks}
\begin{split}
\hh^\ast \BT_{\PP_l}[-1]
&\simeq \hh^\ast\underset{1\leqslant i \leqslant l}{\Boxplus}\RRlHom_q(\SU_i,\SU_i) \\
&\simeq \bigoplus_{1\leqslant i \leqslant l}\hh^\ast \pr_i^\ast q_\ast \RRlHom(\SU_i,\SU_i) \\
&\simeq \bigoplus_{1\leqslant i \leqslant l} \hh^\ast p_\ast (\pr_i \times \id_X)^\ast \RRlHom(\SU_i,\SU_i) \\
&\simeq \bigoplus_{1\leqslant i \leqslant l} \boldit{\pi}_\ast (\hh \times \id_X)^\ast \RRlHom((\pr_i\times\id_X)^\ast\SU_i,(\pr_i\times \id_X)^\ast\SU_i) \\
&\simeq\bigoplus_{1\leqslant i \leqslant l} \RRlHom_{\boldit{\pi}} (\SE_i^{\uni},\SE_i^{\uni}),
\end{split}
\end{equation}
where in the third and fourth equivalences we used the fact that all projections $\pi$, $p$ and $q$ in Diagram \ref{double-fibre-diagram} are obviously schematic and quasicompact in the sense of \cite{studyindag1}, and hence we can safely apply the derived base change formula \cite[Ch.~$3$, Prop.~$2.2.2$]{studyindag1} along \eqref{double-fibre-diagram}. 

Now, the universal maps in \eqref{universal flag} yield natural morphisms
\[
\begin{tikzcd}[column sep=large,row sep=tiny]
\psi_i^\ast \colon \RRlHom_{\boldit{\pi}} (\SE_i^{\uni},\SE_i^{\uni}) \arrow{r}{-\circ\psi_i} & \RRlHom_{\boldit{\pi}} (\SE_{i-1}^{\uni},\SE_i^{\uni}) & i=1,\ldots,l\\
\psi_{i\ast}\colon \RRlHom_{\boldit{\pi}} (\SE_{i-1}^{\uni},\SE_{i-1}^{\uni}) \arrow{r}{\psi_{i}\circ -} & \RRlHom_{\boldit{\pi}} (\SE_{i-1}^{\uni},\SE_{i}^{\uni})  & i=1,\ldots,l
\end{tikzcd}
\]
which, by the chain of equivalences \eqref{chain-pullbacks}, can be assembled to form a morphism 
\begin{equation}\label{Upsilon-map}
\begin{tikzcd}[row sep=tiny]
\hh^\ast \BT_{\PP_l}[-1] \arrow{r}{\Upsilon} & \displaystyle \bigoplus_{1\leqslant i \leqslant l} \RRlHom_{\boldit{\pi}}(\SE_{i-1}^{\uni},\SE_{i}^{\uni}),
\end{tikzcd}
\end{equation}
represented by the $(l-1)\times l$ matrix\footnote{Note that $\psi_{1\ast}$ is not used.}
\[
\Upsilon = 
\begin{pmatrix}
    \psi_1^\ast & 0 & 0 & 0 & \cdots & 0 & 0\\
    -\psi_{2\ast} & \psi_2^\ast  & 0 & 0 & \cdots & 0 & 0 \\
    0 & -\psi_{3\ast} & \psi_3^\ast  & 0 & \cdots & 0 & 0 \\
    \vdots & \vdots & \vdots & \vdots & \cdots & \vdots & \vdots \\
    0 & 0 & 0 & 0 & -\psi_{l-1,\ast} & \psi_{l-1}^\ast & 0 \\
    0 & 0 & 0 & 0 & 0 & -\psi_{l\ast} & \psi_l^\ast
\end{pmatrix}.
\]
Consider the $l$-fold source-and-target morphism
\[
\begin{tikzcd}[column sep=large]
\PP_{[l]} = \DPerf^{[1]}(X)^{\times l}_{\fl} \arrow{rr}{u_l =\langle s,t\rangle^{\times l}} & & \PP_{2l} = \PP_2^{\times l}
\end{tikzcd}
\]
sending $(\CF_i \to \CG_i)_{i=1,\ldots,l} \mapsto (\CF_1,\CG_1,\ldots,\CF_l,\CG_l)$. Denote by $v_i\colon \SF_i \to \SG_{i}$ the universal morphism over the $i$-th copy $\PP_{[1]} \times X$. If $(\SU_i^1,\SU_i^2)$ denotes the universal pair over the $i$-th copy of $\PP_2 \times X$, then it  is clear that $(u_1 \times \id_X)^\ast \SU_i^1 = \SF_i$ and $(u_1 \times \id_X)^\ast \SU_i^2 = \SG_{i}$. Let $\tau\colon \PP_{[1]} \times X \to \PP_{[1]}$ be the projection.
The transitivity triangle of tangent complexes 
\begin{equation}\label{tangent-u} 
\begin{tikzcd}[column sep=large]
u_l^\ast \BT_{\PP_{2l}}[-1] \arrow{r} & 
\BT_{u_l} \arrow{r} & 
\BT_{\PP_{[l]}}
\end{tikzcd}
\end{equation}
attached to the morphism $u_l$ realises $\BT_{\PP_{[l]}}$ as the cone of the morphism $\Gamma$ appearing in the commutative diagram 
\begin{equation}\label{map-gamma}
\begin{tikzcd}[column sep=large,row sep=large]
u_l^\ast \BT_{\PP_{2l}}[-1] \arrow{r}\arrow[d,"\simeq",sloped] & 
\BT_{u_l}\arrow[d,"\simeq",sloped]\\
\displaystyle\underset{1\leqslant i \leqslant l}{\Boxplus}\RRlHom_\tau(\SF_i,\SF_i)\oplus \RRlHom_\tau(\SG_{i},\SG_{i})\arrow{r}{\Gamma} & \displaystyle\underset{1\leqslant i \leqslant l}{\Boxplus}\RRlHom_\tau(\SF_i,\SG_{i}) 
\end{tikzcd}
\end{equation}
where the rightmost vertical isomorphism in \eqref{map-gamma} comes from \Cref{prop:TANGENT-source-and-target}, while $\Gamma$ is defined  by assembling together the natural maps
\[
\begin{tikzcd}[column sep=large,row sep=tiny]
v_i^\ast \colon \RRlHom_\tau(\SG_i,\SG_i)\arrow{r}{-\circ v_i} & \RRlHom_\tau(\SF_i,\SG_i) & i=1,\ldots,l\\
v_{i\ast} \colon \RRlHom_\tau(\SF_i,\SF_i)\arrow{r}{v_i\circ -} & \RRlHom_\tau(\SF_i,\SG_i) & i=1,\ldots,l
\end{tikzcd}
\]
into the $l\times 2l$ matrix
\[
\Gamma
= 
\begin{pmatrix}
v_1^\ast & -v_{1\ast} & 0 & 0 & \cdots & 0 & 0 & 0 \\
0 & 0 & v_2^\ast & -v_{2\ast} & \cdots & 0 & 0 & 0 \\
\vdots & \vdots & \vdots & \vdots & \cdots & \vdots & \vdots & \vdots \\
0 & 0 & 0 & 0 & 0 & 0 & v_l^\ast & -v_{l\ast}
\end{pmatrix}.
\]
Define the map 
\[
\begin{tikzcd}[row sep=tiny]
 \DPerf^{[l]}(X)_{\fl}  \arrow{r}{m_l} & \PP_{[l]}\\
{\{}\CE_0 \to \CE_1\to \CE_2 \to \cdots \to \CE_l{\}} \arrow[mapsto]{r} & (\CE_0 \to \CE_1, \CE_1 \to \CE_2, \ldots, \CE_{l-1} \to \CE_l).
\end{tikzcd}
\]
We have a commutative diagram\footnote{As explained to us by B.~Fantechi \cite{Fantechi-private}.}
\begin{equation}\label{diag:double-cartesian}
\begin{tikzcd}[row sep=large,column sep=huge]
\QQ\arrow[bend left=20]{rrr}[description]{f_\CE}\arrow[swap]{d}{\hh}\arrow[hook]{r}{j_{\CE}} & \DPerf^{[l]}(X)_{\CE/,_{\fl}}\MySymb{dr}\arrow[hook]{r}{i_{\CE}}\arrow{d}{\widetilde u_l} &
\DPerf^{[l]}(X)_{\fl}\MySymb{dr}\arrow{r}{m_l}\arrow{d} & \PP_{[l]}\arrow{d}{u_l} \\
\PP_l \arrow{r}{\sim}\arrow[bend right=20]{rrr}[description]{g_\CE} & \set{\CE} \times \PP_l \arrow[hook]{r} & \PP_{l+1} \arrow{r}{\langle\id,\Delta^{l-1},\id\rangle} & \PP_{2l}
\end{tikzcd}
\end{equation}
where the middle and rightmost squares are cartesian, $\Delta^{l-1}$ denotes $(l-1)$-fold product of the diagonal morphism $\PP_1 \to \PP_1\times\PP_1$ and $j_{\CE}$ is the open immersion of \Cref{prop:openimmersion}.
Consider the morphism
\[
\begin{tikzcd}
f_\CE \times \id_X \colon \QQ \times X \arrow{r} & \PP_{[l]} \times X.
\end{tikzcd}
\]
One has the relations
\begin{equation}\label{pullback-g_E}
\begin{split}
 (f_\CE \times \id_X)^\ast \SF_i &\simeq \SE_{i-1}^{\uni}  \\
 (f_\CE \times \id_X)^\ast \SG_i &\simeq \SE_{i}^{\uni}
\end{split}
\end{equation}
for $1\leqslant i\leqslant l$, where $\SE_0^{\uni} = \CE_{\QQ}$ as in \eqref{universal flag} and, with a slight abuse of notation, we are writing $\SF_i$ and $\SG_i$ for their pullback along the $i$-th projection $\PP_{[l]}\times X \to \PP_{[1]}\times X$.
It follows that, pulling back the exact triangle \eqref{tangent-u} along $f_\CE$ and combining \eqref{pullback-g_E} with the identifications from \eqref{map-gamma} one gets
\begin{equation}\label{pullback-diagram}
\begin{tikzcd}[row sep=large]
f_\CE^\ast u_l^\ast\BT_{\PP_{2l}}[-1] \arrow{r}\arrow[d,"\simeq",sloped] & f_\CE^\ast \BT_{u_l} \arrow[d,"\simeq",sloped] \\
\displaystyle\bigoplus_{1\leqslant i \leqslant l} \RRlHom_{\boldit{\pi}}(\SE_{i-1}^{\uni},\SE_{i-1}^{\uni})\oplus \RRlHom_{\boldit{\pi}}(\SE_{i}^{\uni},\SE_{i}^{\uni})\arrow{r}{f_\CE^\ast \Gamma} & \displaystyle\bigoplus_{1\leqslant i \leqslant l} \RRlHom_{\boldit{\pi}}(\SE_{i-1}^{\uni},\SE_{i}^{\uni}) 
\end{tikzcd}
\end{equation}
with cone $f_\CE^\ast \BT_{\PP_{[l]}}$.

We are now ready to prove \Cref{MAIN-THEOREM-TG-CPLX} from the introduction.

\begin{theorem}\label{prop:tang-cplx-Q}
The absolute tangent complex $\BT_{\QQ}$ is given as
\[
\BT_{\QQ} \simeq \Cone\left(\bigoplus_{1\leqslant i \leqslant l} \RRlHom_{\boldit{\pi}} (\SE_i^{\uni},\SE_i^{\uni})\xrightarrow{~~~\Upsilon~~~}  \displaystyle\bigoplus_{1\leqslant i \leqslant l} \RRlHom_{\boldit{\pi}}(\SE_{i-1}^{\uni},\SE_{i}^{\uni})\right).
\]
\end{theorem}

\begin{proof}
The result follows by combining the following two statements with the transitivity triangle of tangent complexes attached to $\hh$.
The first statement is that there is an equivalence
\begin{equation}\label{T_h}
\begin{tikzcd}
\BT_{\hh} \arrow{r}{\sim} &  \displaystyle\bigoplus_{1\leqslant i \leqslant l} \RRlHom_{\boldit{\pi}}(\SE_{i-1}^{\uni},\SE_{i}^{\uni}).
\end{tikzcd}
\end{equation}
This follows from the fact that all our derived stacks are locally geometric, hence by \Cref{lemma:functoriality-tangents} there is a canonical isomorphism
\[
\begin{tikzcd}
\BT_{\widetilde u_l} \arrow{r}{\sim} & (m_l\circ i_\CE)^\ast \BT_{u_l}.
\end{tikzcd}
\]
On the other hand, $\BT_{u_l} \simeq \Boxplus_{1\leqslant i \leqslant l}\RRlHom_\tau(\SF_i,\SG_{i})$ and therefore 
\[
\begin{tikzcd}
    \BT_{\hh} \simeq j_\CE^\ast \BT_{\widetilde u_l} \arrow{r}{\sim} & f_\CE^\ast \BT_{u_l} \simeq \displaystyle\bigoplus_{1\leqslant i \leqslant l} \RRlHom_{\boldit{\pi}}(\SE_{i-1}^{\uni},\SE_{i}^{\uni}).
\end{tikzcd}
\]
The second statement is that under the identification \eqref{T_h}, the map $\Upsilon$ coincides with the canonical morphism $\hh^\ast \BT_{\PP_l}[-1] \to \BT_{\hh}$ from the transitivity triangle attached to $\hh$. Indeed, consider  the transitivity triangle 
\[
\begin{tikzcd}
\BT_{\PP_l} \arrow{r}{g_{\CE\ast}} & g_\CE^\ast \BT_{\PP_{2l}}\arrow{r} & \BT_{g_\CE}[1].
\end{tikzcd}
\]
The compatibility between tangent complexes along the diagram \eqref{diag:double-cartesian} yields a morphism of exact triangles
\[
\begin{tikzcd}[row sep=large]
& \hh^\ast \BT_{\PP_l}[-1]\arrow{r}\arrow{d}{\hh^\ast g_{\CE\ast}[-1]} & \BT_{\hh}\arrow{r}\arrow[d,"\simeq",sloped] & \BT_{\QQ}\arrow{d} \\
\hh^\ast g_\CE^\ast \BT_{\PP_{2l}} [-1]\arrow[r,"\simeq"] & f_\CE^\ast u_l^\ast \BT_{\PP_{2l}}[-1] \arrow{r}{\beta} & f_\CE^\ast \BT_{u_l} \arrow{r} & f_\CE^\ast \BT_{\PP_{[l]}} 
\end{tikzcd}
\]
but since the map $\beta$ is identified with $f_\CE^\ast \Gamma$ via \eqref{pullback-diagram}, the composition $\beta \circ \hh^\ast g_{\CE\ast}[-1]$ is naturally identified with $\Upsilon$, as required.
\end{proof}

\subsection{Obstruction theory on the classical hyperquot scheme}\label{sec:proof-of-OT}
In this section we prove \Cref{MAIN-THEOREM-OBS-TH} from the introduction. 

Let 
\[
\begin{tikzcd}
Q=\Quot_X^{[l]}(\CE) \arrow[hook]{r} {\iota} & \QQ=\RQuot_X^{[l]}(\CE)    
\end{tikzcd} 
\]
be the canonical inclusion of the ordinary hyperquot scheme into its derived enhancement. Let $\CE_{\QQ} \onto \SE_\bullet^{\uni}$ be the universal flag over $\RQuot_X^{[l]}(\CE) \times X$. Then 
\begin{align}\label{eqn: universal qut}
      (\iota \times \id_X)^\ast (\CE_{\QQ} \onto \SE_\bullet^{\uni})   \simeq(\CE_{Q} \onto  \ST_\bullet^{\uni})
\end{align}
is the universal flag over $Q \times X$.

We are finally ready to prove  our main result on the classical hyperquot scheme $\Quot_X^{[l]}(\CE)$, namely \Cref{MAIN-THEOREM-OBS-TH} from the introduction. 

\begin{theorem}
\label{main-theorem-obs-th}
Let $X$ be a  projective scheme, $\CE\in \Coh(X)$ a perfect coherent sheaf and $l$ a positive integer. Let $\pi\colon Q \times X \to Q$ be the projection. Then $Q$ admits an obstruction theory
\[
\BE_Q = \Cone \left(
\displaystyle\bigoplus_{1\leqslant i \leqslant l} \RRlHom_{\pi}(\ST_{i-1}^{\uni},\ST_{i}^{\uni})^\vee \xrightarrow{~~~~~} \displaystyle\bigoplus_{1\leqslant i \leqslant l} \RRlHom_{\pi}(\ST_{i}^{\uni},\ST_{i}^{\uni})^\vee
\right)[-1]
\xrightarrow{~~~~~} \BL_Q.
\]
\end{theorem}

\begin{proof}
Let $\iota \colon Q \into \QQ$ be the natural inclusion. We know by \Cref{prop:OT-truncated} that the canonical restriction map
\[
\begin{tikzcd}
\iota^\ast \BL_{\QQ} \arrow{r} & \BL_Q
\end{tikzcd}
\]
is an obstruction theory. By \Cref{prop:tang-cplx-Q}, the cotangent complex of $\QQ$ takes the form  
\[
\BL_{\QQ} \simeq \Cone \left(
\displaystyle\bigoplus_{1\leqslant i \leqslant l} \RRlHom_{\boldit{\pi}}(\SE_{i-1}^{\uni},\SE_{i}^{\uni})^\vee \xrightarrow{~~\Upsilon^\vee~~} 
\displaystyle\bigoplus_{1\leqslant i \leqslant l} \RRlHom_{\boldit{\pi}}(\SE_{i}^{\uni},\SE_{i}^{\uni})^\vee
\right)[-1].
\]
By \eqref{eqn: universal qut}, we have
\[
\iota^\ast \BL_{\QQ} \simeq \BE_Q,
\]
which proves the result.
\end{proof}

\subsection{Hyperquot schemes on curves}
\label{sec:curve-case}
Let $X$ be a smooth projective curve, $\CE \in \Coh(X)$ a locally free sheaf, and $l$ a positive integer. Form the hyperquot scheme
\[
Q=\Quot_X^{[l]}(\CE).
\]
In \cite[Thm.~A]{monavari2024hyperquot}, the first and third author constructed a \emph{perfect} obstruction theory
\[
\BE'_{Q} = \Cone
\left(\displaystyle\bigoplus_{1\leqslant i \leqslant l} \RRlHom_\pi (\CK_i,\CK_i) \xrightarrow{~~~~~~} 
\displaystyle\bigoplus_{1\leqslant i \leqslant l} \RRlHom_\pi (\CK_{i},\CK_{i+1})\right)^\vee 
\xrightarrow{~~\phi'~~} \BL_{Q},
\]
where $\CK_i = \ker (\CE_Q \onto \CT_i^{\uni})$ is the $i$-th universal kernel for $i=1,\ldots,l$ and $\CK_{l+1} = \CE_Q$. Even though the complex $\BE'_Q$ might differ from the complex $\BE_Q$ of \Cref{main-theorem-obs-th}, one can easily check that their K-theory classes match in $K_0(Q)$. Therefore by \cite{Siebert} there is an identity of virtual fundamental classes
\[
[\BE_Q,\phi]^{\vir} = [\BE'_Q,\phi']^{\vir} \in A_\ast (Q),
\]
and, a fortiori, their associated enumerative theories also agree. 

Finally, one immediately sees that, in the special case $l=1$, the two obstruction theories agree, matching the 2-term complex
\[
\RRlHom_\pi(\CK,\CT)^\vee.
\]
This perfect obstruction theory was already constructed by Marian--Oprea \cite{MO_duke} and Oprea \cite{Oprea_zero_locus} for $\CE = \OO_X^{\oplus r}$ and by Gillam \cite[Lemma 4.7]{Gillam}.

\subsection*{Acknowledgements}
We are grateful to Barbara Fantechi for pointing us to the key diagram \eqref{diag:double-cartesian}, and to Qingyuan Jiang and Nick Kuhn for useful discussions.S.M. was supported by the FNS Project 200021-196960 ``Arithmetic aspects of moduli spaces on
curves'', and by  the HORIZON-MSCA-2024-PF-01 Project 
101203281 ``Moduli Spaces of Sheaves: Geometry and Invariants'', funded by the Research and Innovation framework programme Horizon Europe {\normalsize\euflag}. A.R. is partially supported by the PRIN project 2022BTA242 ``Geometry of algebraic structures: moduli, invariants, deformations''. All authors are members of the GNSAGA - INdAM.

\bibliographystyle{amsplain-nodash}
\bibliography{The_Bible}

\ifx\undefined\bysame
\newcommand{\bysame}{\leavevmode\hbox to3em{\hrulefill}\,}
\fi
\begin{thebibliography}{10}

\bibitem{DerivedQuotNew}
Nachiketa Adhikari, {\em Derived quot schemes}, \href{https://arxiv.org/abs/2211.13845}{ArXiv:2211.13845}, 2022.

\bibitem{MR4220750}
Noah Arbesfeld, Drew Johnson, Woonam Lim, Dragos Oprea, and Rahul Pandharipande, {\em The virtual {$K$}-theory of {Q}uot schemes of surfaces}, J. Geom. Phys. {\bf 164} (2021), Paper No. 104154, 36.

\bibitem{BR18}
Sjoerd Beentjes and Andrea~T. Ricolfi, {\em {Virtual counts on Quot schemes and the higher rank local DT/PT correspondence}}, {Math. Res. Lett.} {\bf 28} (2021), no.~4, 967--1032.

\bibitem{BFinc}
Kai Behrend and Barbara Fantechi, {\em The intrinsic normal cone}, Invent. Math. {\bf 128} (1997), no.~1, 45--88.

\bibitem{integraltransforms}
David Ben-Zvi, John Francis, and David Nadler, {\em Integral transforms and {D}rinfeld centers in derived algebraic geometry}, J. Amer. Math. Soc. {\bf 23} (2010), no.~4, 909--966.

\bibitem{BFT_fisica}
Giulio Bonelli, Nadir Fasola, and Alessandro Tanzini, {\em Defects, nested instantons and comet-shaped quivers}, Lett. Math. Phys. {\bf 111} (2021), no.~2, 34, 53.

\bibitem{BFT_math}
Giulio Bonelli, Nadir Fasola, and Alessandro Tanzini, {\em Flags of sheaves, quivers and symmetric polynomials}, Forum Math. Sigma {\bf 12} (2024), 51, e74.

\bibitem{BSY_derived1}
Dennis Borisov, Ludmil Katzarkov, and Artan Sheshmani, {\em Shifted symplectic structures on derived {Q}uot-stacks {I} -- {D}ifferential graded manifolds}, Adv. Math. {\bf 403} (2022), 108369, 31.

\bibitem{borisov2023shiftedsymplecticstructuresderived}
Dennis Borisov, Ludmil Katzarkov, and Artan Sheshmani, {\em Shifted symplectic structures on derived quot-stacks. {II}: {Derived} quot-schemes as dg manifolds}, Adv. Math. {\bf 462} (2025), 22, Id/No 110092.

\bibitem{cazzaniga2020higher}
Alberto Cazzaniga, Dimbinaina Ralaivaosaona, and Andrea~T. Ricolfi, {\em {Higher rank motivic Donaldson-Thomas invariants of $\mathbb{A}^3$ via wall-crossing, and asymptotics}}, Math. Proc. Cambridge Philos. Soc. {\bf 174} (2023), no.~1, 97--122.

\bibitem{cazzaniga2020framed}
Alberto Cazzaniga and Andrea~T. Ricolfi, {\em {Framed sheaves on projective space and Quot schemes}}, Math. Z. {\bf 300} (2022), 745--760.

\bibitem{MR2716513}
Jan Cheah, {\em The cohomology of smooth nested {H}ilbert schemes of points}, ProQuest LLC, Ann Arbor, MI, 1994, PhD Thesis, The University of Chicago.

\bibitem{MR1616606}
Jan Cheah, {\em Cellular decompositions for nested {H}ilbert schemes of points}, Pacific J. Math. {\bf 183} (1998), no.~1, 39--90.

\bibitem{MR1612677}
Jan Cheah, {\em The virtual {H}odge polynomials of nested {H}ilbert schemes and related varieties}, Math. Z. {\bf 227} (1998), no.~3, 479--504.

\bibitem{Chen-hyperquot}
Linda Chen, {\em Poincar{\'e} polynomials of hyperquot schemes}, Math. Ann. {\bf 321} (2001), no.~2, 235--251.

\bibitem{Chen_quantum}
Linda Chen, {\em Quantum cohomology of flag manifolds}, Adv. Math. {\bf 174} (2003), no.~1, 1--34.

\bibitem{CF_quantum_cohom_IMRN}
Ionu\c{t} Ciocan-Fontanine, {\em Quantum cohomology of flag varieties}, Int. Math. Res. Not. IMRN (1995), no.~6, 263--277.

\bibitem{CF_quantum_duke}
Ionu\c{t} Ciocan-Fontanine, {\em On quantum cohomology rings of partial flag varieties}, Duke Math. J. {\bf 98} (1999), no.~3, 485--524.

\bibitem{CF_quantum_Trans}
Ionu\c{t} Ciocan-Fontanine, {\em The quantum cohomology ring of flag varieties}, Trans. Amer. Math. Soc. {\bf 351} (1999), no.~7, 2695--2729.

\bibitem{DerivedQuot}
Ionu\c{t} Ciocan-Fontanine and Mikhail Kapranov, {\em Derived {Q}uot schemes}, Ann. Sci. \'{E}cole Norm. Sup. (4) {\bf 34} (2001), no.~3, 403--440.

\bibitem{DerivedHilbert}
Ionu\c{t} Ciocan-Fontanine and Mikhail Kapranov, {\em Derived {H}ilbert schemes}, J. Amer. Math. Soc. {\bf 15} (2002), no.~4, 787--815.

\bibitem{DavisonR}
Ben Davison and Andrea~T. Ricolfi, {\em {The local motivic DT/PT correspondence}}, J. Lond. Math. Soc. {\bf 104} (2021), no.~3, 1384--1432.

\bibitem{Fantechi-private}
Barbara Fantechi, {\em private communication}.

\bibitem{fga_explained}
Barbara Fantechi, Lothar G\"{o}ttsche, Luc Illusie, Steven~L. Kleiman, Nitin Nitsure, and Angelo Vistoli, {\em Fundamental algebraic geometry}, Mathematical Surveys and Monographs, vol. 123, American Mathematical Society, Providence, RI, 2005, Grothendieck's FGA explained.

\bibitem{fasola2023tetrahedron}
Nadir Fasola and Sergej Monavari, {\em Tetrahedron instantons in {Donaldson}-{Thomas} theory}, Adv. Math. {\bf 462} (2025), 47, Id/No 110099.

\bibitem{FMR_K-DT}
Nadir Fasola, Sergej Monavari, and Andrea~T. Ricolfi, {\em {Higher rank K-theoretic Donaldson--Thomas theory of points}}, Forum Math. Sigma {\bf 9} (2021), no.~E15, 1--51.

\bibitem{studyindag1}
Dennis Gaitsgory and Nick Rozenblyum, {\em A study in derived algebraic geometry. {V}ol. {I}. {C}orrespondences and duality}, Mathematical Surveys and Monographs, vol. 221, American Mathematical Society, Providence, RI, 2017.

\bibitem{studyindag2}
Dennis Gaitsgory and Nick Rozenblyum, {\em A study in derived algebraic geometry. {V}ol. {II}. {D}eformations, {L}ie theory and formal geometry}, Mathematical Surveys and Monographs, vol. 221, American Mathematical Society, Providence, RI, 2017.

\bibitem{zbMATH07184851}
Amin Gholampour, Artan Sheshmani, and Shing-Tung Yau, {\em Nested {Hilbert} schemes on surfaces: virtual fundamental class}, Adv. Math. {\bf 365} (2020), 50, Id/No 107046.

\bibitem{zbMATH07159378}
Amin Gholampour and Richard~P. Thomas, {\em Degeneracy loci, virtual cycles and nested {Hilbert} schemes. {I}}, Tunis. J. Math. {\bf 2} (2020), no.~3, 633--665.

\bibitem{zbMATH07270076}
Amin Gholampour and Richard~P. Thomas, {\em Degeneracy loci, virtual cycles and nested {Hilbert} schemes. {II}}, Compos. Math. {\bf 156} (2020), no.~8, 1623--1663.

\bibitem{Gillam}
William~D. Gillam, {\em {Deformation of quotients on a product}}, \href{https://arxiv.org/abs/1103.5482}{ArXiv:1103.5482}, 2011.

\bibitem{double-nested-1}
Michele Graffeo, Paolo Lella, Sergej Monavari, Andrea~T. Ricolfi, and Alessio Sammartano, {\em The geometry of double nested {Hilbert} schemes of points on curves}, Trans. Am. Math. Soc. {\bf 378} (2025), no.~9, 6013--6047.

\bibitem{Grothendieck_Quot}
Alexander Grothendieck, {\em Techniques de construction et th\'{e}or{\`e}mes d'existence en g\'{e}om\'{e}trie alg\'{e}brique. {IV}. {L}es sch\'{e}mas de {H}ilbert}, S\'{e}minaire {B}ourbaki, Exp. No. 221, no.~6, Soc. Math. France, 1961, pp.~249--276.

\bibitem{Hirschowitz1998DescentePL}
Andr{\'e} Hirschowitz and Carlos Simpson, {\em Descente pour les n-champs (descent for n-stacks)}, \href{https://arxiv.org/abs/math/9807049}{ArXiv:9807049}, 1998.

\bibitem{modulisheaves}
Daniel Huybrechts and Manfred Lehn, {\em The geometry of moduli spaces of sheaves}, Second Edition, Cambridge University Press, 2010, pp.~xviii+325.

\bibitem{derived-grassmannians-schur}
Qingyuan Jiang, {\em {Derived Grassmannians and derived Schur functors}}, \href{https://arxiv.org/abs/2212.10488}{ArXiv:2212.10488}, 2023.

\bibitem{derived-projectivizations}
Qingyuan Jiang, {\em Derived projectivizations of complexes}, To appear in Mem. Amer. Math. Soc., \href{https://arxiv.org/abs/2202.11636}{ArXiv:2202.11636}, 2023.

\bibitem{zbMATH07344916}
Drew Johnson, Dragos Oprea, and Rahul Pandharipande, {\em Rationality of descendent series for {Hilbert} and {Quot} schemes of surfaces}, Sel. Math., New Ser. {\bf 27} (2021), no.~2, 52, Id/No 23.

\bibitem{kontsevich_94}
Maxim Kontsevich, {\em Enumeration of rational curves via torus actions}, The moduli space of curves ({T}exel {I}sland, 1994), Progr. Math., vol. 129, Birkh\"{a}user Boston, Boston, MA, 1995, pp.~335--368.

\bibitem{Lim:2020aa}
Woonam Lim, {\em Virtual {{\(\chi_{-y}\)}}-genera of {Quot} schemes on surfaces}, J. Lond. Math. Soc., II. Ser. {\bf 104} (2021), no.~3, 1300--1341.

\bibitem{htt}
Jacob Lurie, {\em Higher topos theory}, Annals of Mathematics Studies, vol. 170, Princeton University Press, Princeton, NJ, 2009.

\bibitem{dagviii}
Jacob Lurie, {\em {Derived Algebraic Geometry VIII: Quasi-Coherent Sheaves and Tannaka Duality Theorems}}, \href{https://www.math.ias.edu/~lurie/papers/DAG-VIII.pdf}{Online source}, 2011.

\bibitem{ha}
Jacob Lurie, {\em Higher algebra}, \href{https://people.math.harvard.edu/~lurie/papers/HA.pdf}{Online source}, 2017, p.~1553.

\bibitem{sag}
Jacob Lurie, {\em Spectral algebraic geometry (under construction)}, \href{https://www.math.ias.edu/~lurie/papers/SAG-rootfile.pdf}{Online source}, 2018.

\bibitem{MO_duke}
Alina Marian and Dragos Oprea, {\em Virtual intersections on the {Q}uot scheme and {V}afa-{I}ntriligator formulas}, Duke Math. J. {\bf 136} (2007), no.~1, 81--113.

\bibitem{Mon_double_nested}
Sergej Monavari, {\em Double nested {H}ilbert schemes and the local stable pairs theory of curves}, Compos. Math. {\bf 158} (2022), no.~9, 1799--1849.

\bibitem{MR_nested_Quot}
Sergej Monavari and Andrea~T. Ricolfi, {\em On the motive of the nested {Q}uot scheme of points on a curve}, J. Algebra {\bf 610} (2022), 99--118.

\bibitem{Lissite-quot}
Sergej Monavari and Andrea~T. Ricolfi, {\em Sur la lissit{\'e} du sch{\'e}ma {Quot} ponctuel embo{\^\i}t{\'e}}, Can. Math. Bull. {\bf 66} (2023), no.~1, 178--184.

\bibitem{monavari2024hyperquot}
Sergej Monavari and Andrea~T. Ricolfi, {\em Hyperquot schemes on curves: virtual class and motivic invariants}, Math. Ann. {\bf 392} (2025), no.~2, 1665--1709.

\bibitem{Neeman2}
Amnon Neeman, {\em The {G}rothendieck duality theorem via {B}ousfield's techniques and {B}rown representability}, J. Amer. Math. Soc. {\bf 9} (1996), no.~1, 205--236.

\bibitem{Nit}
Nitin Nitsure, {\em Construction of {H}ilbert and {Q}uot schemes}, Fundamental Algebraic Geometry, Math. Surveys Monogr., vol. 123, Amer. Math. Soc., Providence, RI, 2005, pp.~105--137.

\bibitem{Olsson}
Martin Olsson, {\em Sheaves on {A}rtin stacks}, J. Reine Angew. Math. {\bf 603} (2007), 55--112.

\bibitem{Oprea_zero_locus}
Dragos Oprea, {\em Notes on the moduli space of stable quotients}, In: Compactifying moduli spaces, Adv. Courses Math. CRM Barcelona, Birkh\"{a}user/Springer, Basel, 2016, pp.~69--135.

\bibitem{Pandit2011ModuliPI}
Pranav~S. Pandit, {\em Moduli problems in derived noncommutative geometry}, PhD thesis, University of Pennsylvania, \href{https://repository.upenn.edu/server/api/core/bitstreams/18e8493e-c877-447d-aeb5-75ab8b332812/content}{Online source}, 2011.

\bibitem{Poma-VCF-Artin}
Flavia Poma, {\em Virtual classes of {Artin} stacks}, Manuscr. Math. {\bf 146} (2015), no.~1-2, 107--123.

\bibitem{Porta_Sala_Hall_algebras}
Mauro Porta and Francesco Sala, {\em Two-dimensional categorified {H}all algebras}, J. Eur. Math. Soc. (JEMS) {\bf 25} (2023), no.~3, 1113--1205.

\bibitem{Ricolfi2018}
Andrea~T. Ricolfi, {\em The {DT}/{PT} correspondence for smooth curves}, Math. Z. {\bf 290} (2018), no.~1-2, 699--710.

\bibitem{LocalDT}
Andrea~T. Ricolfi, {\em {Local contributions to Donaldson--Thomas invariants}}, Int. Math. Res. Not. IMRN {\bf 2018} (2018), no.~19, 5995--6025.

\bibitem{ricolfi2019motive}
Andrea~T. Ricolfi, {\em {On the motive of the Quot scheme of finite quotients of a locally free sheaf}}, J. Math. Pures Appl. {\bf 144} (2020), 50--68.

\bibitem{Virtual_Quot}
Andrea~T. Ricolfi, {\em {Virtual classes and virtual motives of Quot schemes on threefolds}}, Adv. Math. {\bf 369} (2020), 107182.

\bibitem{d-critical_quot}
Andrea~T. Ricolfi and Michail Savvas, {\em {The d-critical structure on the Quot scheme of points of a Calabi--Yau 3-fold}}, Commun. Contemp. Math. {\bf 26} (2024), no.~8, 1--45.

\bibitem{Schurg_Toen_Vezzosi_Determinants_of_perfect_complexes}
Timo Sch\"urg, Bertrand To\"en, and Gabriele Vezzosi, {\em Derived algebraic geometry, determinants of perfect complexes, and applications to obstruction theories for maps and complexes}, J. Reine Angew. Math. {\bf 702} (2015), 1--40.

\bibitem{Sernesi}
Edoardo Sernesi, {\em Deformations of algebraic schemes}, Grundlehren der Mathematischen Wissenschaften, vol. 334, Springer-Verlag, Berlin, 2006.

\bibitem{Sheshmani-Yau}
Artan Sheshmani and Shing-Tung Yau, {\em Higher rank flag sheaves on surfaces}, Eur. J. Math. {\bf 44} (2024), no.~3, 1--44.

\bibitem{Siebert}
Bernd Siebert, {\em Virtual fundamental classes, global normal cones and {F}ulton's canonical classes}, Frobenius manifolds, Aspects Math., vol.~36, Friedr. Vieweg, Wiesbaden, 2004, pp.~341--358.

\bibitem{stacks-project}
{Stacks Project Authors}, {\em {\href{http://stacks.math.columbia.edu}{Stacks Project}}}, 2024.

\bibitem{stark2024quot}
Samuel Stark, {\em {On the Quot scheme $\mathrm{Quot}^{l}_{S}(\mathcal{E})$}}, {\href{https://arxiv.org/abs/2107.03991}{ArXiv:2107.03991}}, 2021.

\bibitem{Toen-higher-and-derived}
Bertrand To{\"e}n, {\em Higher and derived stacks: a global overview}, Algebraic geometry, Seattle 2005. Proceedings of the 2005 Summer Research Institute, Seattle, WA, USA, July 25--August 12, 2005, Providence, RI: American Mathematical Society (AMS), 2009, pp.~435--487.

\bibitem{moduli-of-objects-dg}
Bertrand To{\"e}n and Michel Vaqui{\'e}, {\em Moduli of objects in dg-categories}, Ann. Sci. {\'E}c. Norm. Sup{\'e}r. (4) {\bf 40} (2007), no.~3, 387--444.

\bibitem{hag1}
Bertrand To\"{e}n and Gabriele Vezzosi, {\em Homotopical algebraic geometry. {I}. {T}opos theory}, Adv. Math. {\bf 193} (2005), no.~2, 257--372.

\bibitem{hag2}
Bertrand To\"{e}n and Gabriele Vezzosi, {\em Homotopical algebraic geometry. {II}. {G}eometric stacks and applications}, Mem. Amer. Math. Soc. {\bf 193} (2008), no.~902, x+224.

\end{thebibliography}

\medskip
\noindent
{\small{Sergej Monavari \\
\address{Dipartimento di Matematica “Tullio Levi-Civita”, Università degli Studi di Padova, Via Trieste 63, 35121 Padova (Italy)} \\
\href{mailto:sergej.monavari@unipd.it}{\texttt{sergej.monavari@unipd.it}}
}}

\bigskip
\noindent
{\small Emanuele Pavia \\
\address{Maison du Nombre, University of Luxembourg, 6, avenue de la Fonte L-4364, Esch-sur-Alzette (Luxembourg)} \\
\href{emanuele.pavia@uni.lu}{\texttt{emanuele.pavia@uni.lu}}}

\bigskip
\noindent
{\small Andrea T. Ricolfi \\
\address{SISSA, Via Bonomea 265, 34136, Trieste (Italy)} \\
\href{mailto:aricolfi@sissa.it}{\texttt{aricolfi@sissa.it}}}
\end{document}